\documentclass[twoside]{article}

\usepackage{amsmath}
\usepackage{amssymb}
\usepackage[numbers,authoryear]{natbib}
\usepackage{tab}

\usepackage{mor}

\received{}                              %
\revised{}                               %
\pubyear{0x}                             %
\pubmonth{Xxxxxxx}                       %
\volume{xx}                              %
\issue{x}                                %
\pages{xxx--xxx}                         %
\firstpage{0xxx}                         %
\DOI{10.1287/moor.xxxx.xxxx}             %
\startpagenumber{1}                      %

\title{Robust approachability and regret minimization in games with partial monitoring}
\ShortTitle{Robust approachability and regret minimization in games with partial monitoring}
\ShortAuthors{Mannor, Perchet, and Stoltz}
\NumberOfAuthors{3}
\FirstAuthor{Shie Mannor}
\FirstAuthorAddress{Israel Institute of Technology (Technion), Haifa, Israel}
\FirstAuthorEmail{shie@ee.technion.ac.il}
\FirstAuthorURL{http://webee.technion.ac.il/people/shie/}
\SecondAuthor{Vianney Perchet}
\SecondAuthorAddress{Universit{\'e} Paris-Diderot, Paris, France}
\SecondAuthorEmail{vianney.perchet@normalesup.org}
\SecondAuthorURL{https://sites.google.com/site/vianneyperchet/home}
\ThirdAuthor{Gilles Stoltz}
\ThirdAuthorAddress{Ecole Normale Sup{\'e}rieure -- CNRS -- INRIA, Paris, France \& HEC Paris -- CNRS, Jouy-en-Josas, France}
\ThirdAuthorEmail{gilles.stoltz@ens.fr}
\ThirdAuthorURL{http://www.math.ens.fr/~stoltz}


\MSCcodes{Primary: 91A20, 91A26, 62L12; Secondary: 68Q32} 

\ORMScodes{Primary: decision analysis: sequential; games/group decisions: noncooperative; secondary: computer science: artificial intelligence} 

\newcommand{\bB}{\mathbf{B}}
\newcommand{\cA}{\mathcal{A}}
\newcommand{\cB}{\mathcal{B}}
\newcommand{\cI}{\mathcal{I}}
\newcommand{\cJ}{\mathcal{J}}
\newcommand{\cC}{\mathcal{C}}
\newcommand{\cG}{\mathcal{G}}
\newcommand{\cH}{\mathcal{H}}
\newcommand{\cF}{\mathcal{F}}
\newcommand{\cK}{\mathcal{K}}

\newcommand{\cS}{\mathcal{S}\bigl( \R^d \bigr)}
\newcommand{\R}{\mathbb{R}}
\newcommand{\E}{\mathbb{E}}
\newcommand{\ol}{\overline}
\renewcommand{\mathbf}{\boldsymbol}

\newcommand{\bp}{\mathbf{p}}
\newcommand{\bq}{\mathbf{q}}
\newcommand{\bx}{\mathbf{x}}
\newcommand{\by}{\mathbf{y}}
\newcommand{\bu}{\mathbf{u}}
\newcommand{\bv}{\mathbf{v}}

\newcommand{\tH}{\widetilde{H}}
\newcommand{\norm}[1][\,\cdot\,]{\ensuremath{\left\Arrowvert #1 \right\Arrowvert}}
\newcommand{\bnorm}[1][\,\cdot\,]{\ensuremath{\bigl\Arrowvert #1 \bigr\Arrowvert}}
\newcommand{\Bnorm}[1][\,\cdot\,]{\ensuremath{\Bigl\Arrowvert #1 \Bigr\Arrowvert}}
\newcommand{\om}{\overline{m}}

\renewcommand{\epsilon}{\varepsilon}
\newcommand{\eps}{\varepsilon}
\renewcommand{\leq}{\leqslant}
\renewcommand{\geq}{\geqslant}

\renewcommand{\hat}{\widehat}
\newcommand{\wh}{\widehat}
\newcommand{\e}{\varepsilon}
\newcommand{\RAC}{\mbox{RAC}}
\newcommand{\oom}{\overline{\overline{m}}}
\newcommand{\ind}{\mathbb{I}}
\newcommand{\steps}[2]{\textbf{Step #1: Assertion}~(\ref{#2})\textbf{.}~~}
\newcommand{\ext}{\mbox{\rm \scriptsize{ext}}}
\newcommand{\cone}{\mbox{\rm \scriptsize{cone}}}
\newcommand{\inter}{\mbox{\rm \scriptsize{int}}}
\newcommand{\swap}{\mbox{\rm \scriptsize{\rm swap}}}
\newcommand{\bth}{\mathbf{\theta}}
\DeclareMathOperator{\co}{co}
\newcommand{\tm}{\widetilde{m}}
\newcommand{\tr}{\widetilde{r}}
\renewcommand{\P}{\mathbb{P}}
\newcommand{\apm}{\eqref{eq:APM}}
\newcommand{\APM}{\mbox{APM}}
\newcommand{\sigdark}{\clubsuit}
\newcommand{\sigother}{\heartsuit}
\newcommand{\wt}{\widetilde}

\newcommand{\citemor}[1]{\citeauthor{#1}~\cite{#1}}
\newtheorem{condition}{\sc Condition}

\begin{document}
\setcitestyle{numbers}

\maketitle

\begin{abstract}
Approachability has become a standard tool in analyzing learning algorithms in the adversarial online learning setup.
We develop a variant of approachability for games where there is ambiguity in the obtained reward that belongs to a set, rather than being a single vector. Using this variant we tackle the problem of approachability in games with partial monitoring and develop simple and efficient
algorithms (i.e., with constant per-step complexity) for this setup.
We finally consider external regret and internal regret in repeated games with partial monitoring and  derive
regret-minimizing strategies based on approachability theory.
\end{abstract}
\normalsize

\section{Introduction.}

Blackwell's approachability theory and its variants has become a standard and useful tool in analyzing online learning algorithms (\citemor{CBL06})
and algorithms for learning in games (\citeauthor{HM00}~\cite{HM00,HM01}). The first application of Blackwell's approachability to learning in the online setup is due to Blackwell~\cite{Bla56b} himself.
Numerous other contributions are summarized in the monograph by~\citemor{CBL06}.
Blackwell's approachability theory enjoys a  clear geometric interpretation that allows it to be used in situations
where online convex optimization or exponential weights do not seem to be easily applicable and, in some sense,
to go beyond the minimization of the regret and/or to control quantities of a different flavor;
e.g., in the article by~\citemor{MannorTsitsiklisYu09}, to minimize the regret together with path constraints, and in
the one by~\citemor{MShimkin08Regret}, to minimize the regret in games whose stage duration is not fixed.
Recently, it has been shown by~\citemor{ABH10} that approachability and low regret learning are
equivalent in the sense that efficient reductions exist from one to the other.
Another recent paper by~\citemor{RST10} showed that approachability
can be analyzed from the perspective of learnability using tools from learning theory.


In this paper we consider approachability and online learning with partial monitoring in games against Nature. In partial monitoring the decision maker does not know how much reward was obtained and only gets a (random) signal whose distribution depends on the action of the decision maker and the action of Nature. There are two extremes of this setup that are well studied. On the one extreme we have the case where the signal includes the reward itself (or a signal that can be used to unbiasedly estimate the reward), which is essentially the celebrated bandits setup.
The other extreme is the case where the signal is not informative (i.e., it tells the decision maker nothing about the actual reward obtained); this setting then essentially consists of repeating the same situation over and over again, as no information
is gained over time. We consider a setup encompassing these situations and more general ones, in which the signal is indicative of the actual reward, but is not necessarily a sufficient statistics thereof.
The difficulty is that the decision maker cannot compute the actual reward he obtained nor the actions of Nature.

Regret minimization with partial monitoring has been studied in several papers in the learning theory community.
\citemor{PiSc01}, \citemor{MaSh03}, \citemor{CeLuSt06}
study special cases where an accurate estimation of the rewards (or worst-case rewards) of the decision maker is possible thanks to some extra structure.
A general policy with vanishing regret is presented by~\citemor{LMS08}. This policy is based on exponential weights and a specific estimation procedure for the (worst-case) obtained rewards. In contrast, we provide approachability-based results for the problem of regret minimization. On route, we define a new type of approachability setup, with enables to re-derive
the extension of approachability to the partial monitoring vector-valued setting proposed by~\citemor{Per11}.
More importantly, we provide concrete algorithms
for this approachability problem that are more efficient in the sense that, unlike previous works in the domain,
their complexity is constant over all steps. Moreover, their rates of convergence are independent of the game at hand, as in
the seminal paper by~\citemor{Bla56b} but for the first time in this general framework. For example, the recent purely theoretical (and fairly technical) study of approachability
\citemor{PQ12}, which is based on somehow related arguments, does neither provide rates of convergence nor concrete algorithms for this matter.

\paragraph{Outline.} The paper is organized as follows.
In Section \ref{se:basics} we recall some basic facts from approachability theory in the standard vector-valued games setting where a decision maker is engaged in a repeated vector-valued game against an arbitrary opponent (or ``Nature").
In Section \ref{se:robapproach} we propose a novel setup for approachability, termed ``robust approachability,'' where instead of obtaining a vector-valued reward, the decision maker obtains a set, that represents the ambiguity concerning his reward. We provide a simple characterization of approachable convex sets and an algorithm for the set-valued reward setup under the assumption that the set-valued reward functions are linear. In Section \ref{sec:robappr-gal} we extend the robust approachability setup to problems where the set-valued reward functions are not linear, but rather concave in the mixed action of the decision maker and convex in the mixed action of Nature.
In Section \ref{se:apptogames} we show how to apply the robust approachability framework to the repeated vector-valued games with partial monitoring.
In Section \ref{sec:main} we consider a special type of games where the signaling structure possesses a special property,
called bi-piecewise linearity, that can be exploited to derive efficient strategies.
This type of games is rich enough as it encompasses several useful special cases.
In Section \ref{se:bipiecewise} we provide a simple and constructive algorithm for theses games.
Previous results for approachability in this setup were either non-constructive (\citemor{Rus99}) or were highly
inefficient as they relied on some sort of lifting to the space of probability measures on mixed actions (\citemor{Per11}) and typically required a grid that is progressively refined
(leading to a step complexity that is exponential in the number $T$ of past steps).
In Section \ref{se:apptoregretmin} we apply our results for both external-regret and internal-regret minimization in repeated games with partial monitoring. In both cases our proofs are simple, lead to algorithms with constant complexity at each step, and are accompanied with rates. Our results for external regret have rates similar to the ones obtained by~\citemor{LMS08},
but our proof is direct and simpler.
In Section \ref{sec:galstruc} we mention the general signaling case and explain how it is possible to approach certain special sets such as polytopes efficiently and general convex sets although inefficiently.

\section{Some basic facts from approachability theory.}
\label{se:basics}

In this section we recall the most basic versions of Blackwell's approachability theorem
for vector-valued payoff functions.

We consider a vector-valued game between two players, a decision maker (first player) and Nature
(second player), with respective finite action sets $\cA$ and $\cB$, whose
cardinalities are referred to as $N_\cA$ and $N_\cB$.
We denote by $d$ the dimension of the reward vector
and equip $\R^d$ with the $\ell^2$--norm $\norm_2$.
The payoff function of the first player is given by
a mapping $m : \cA \times \cB \to \R^d$,
which is multi-linearly extended to $\Delta(\cA) \times \Delta(\cB)$,
the set of product-distributions over $\cA \times \cB$.

We consider two frameworks, depending on whether pure or mixed actions are taken.

\paragraph{Pure actions taken and observed.} We
denote by $A_1,\, A_2, \, \ldots$ and $B_1, \, B_2, \, \ldots$ the actions in $\cA$
and $\cB$ sequentially taken by each player; they are possibly given by randomized strategies,
i.e., the actions $A_t$ and $B_t$ were obtained by random draws according to respective probability
distributions denoted by $\bx_t \in \Delta(\cA)$ and $\by_t \in \Delta(\cB)$.
For now, we assume that the first player has a full or bandit monitoring of the pure actions taken
by the opponent player: at the end of round $t$, when receiving the payoff $m(A_t,B_t)$,
either the pure action $B_t$ (full monitoring) or only the indicated payoff (bandit monitoring)
is revealed to him.

\begin{definition}
\label{def:pure}
A set $\cC \subseteq \R^d$  is {\em $m$--approachable with pure actions} if there exists a
strategy of the first player such that, for all $\eps > 0$,
there exists an integer $T_\eps$ such that for all strategies of the second player,
\[
\P \left\{ \forall \, T \geq T_\eps, \quad \inf_{c \in \cC} \ \norm[c - \frac{1}{T} \sum_{t=1}^T m \bigl( A_t,B_t \bigr)]_2 \leq \eps \right\}
\ \geq 1-\eps\,.
\]
In particular, the first player has a strategy that ensures that the average of his vector-valued payoffs
converges almost surely to the set $\cC$ (uniformly with respect to the strategies of the second player).
\end{definition}

The above convergence will be achieved in the course of this paper under two forms.
Most often we will exhibit strategies such that, for all strategies of the second player,
for all $\delta > 0$, with probability at least $1-\delta$,
\[
\inf_{c \in \cC} \ \norm[c - \frac{1}{T} \sum_{t=1}^T m \bigl( A_t,B_t \bigr)]_2 \leq
\beta(T,\delta)\,.
\]
A union bound shows that such strategies $m$--approach $\cC$
as soon as there exists a positive sequence $\eps_T$ such that $\sum \eps_t$ is finite
and $\beta(T,\eps_T) \to 0$. Sometimes we will also deal with strategies directly ensuring
that, for all strategies of the second player,
for all $\delta > 0$, with probability at least $1-\delta$,
\[
\sup_{\tau \geq T} \ \inf_{c \in \cC} \ \norm[c - \frac{1}{\tau} \sum_{t=1}^\tau m \bigl( A_t,B_t \bigr)]_2 \leq
\beta(T,\delta)\,.
\]
Such strategies $m$--approach $\cC$ as soon as $\beta(T,\delta) \to 0$ for all $\delta > 0$.

\paragraph{Mixed actions taken and observed.} In
this case, we denote by $\bx_1,\, \bx_2, \, \ldots$ and $\by_1, \, \by_2, \, \ldots$ the actions in $\Delta(\cA)$
and $\Delta(\cB)$ sequentially taken by each player. We also assume a full or bandit monitoring for the first player:
at the end of round $t$, when receiving the payoff $m(\bx_t,\by_t)$,
either the mixed action $\by_t$ (full monitoring) or the indicated payoff (bandit monitoring)
is revealed to him.

\begin{definition}
A set $\cC \subseteq \R^d$  is {\em $m$--approachable with mixed actions} if there exists a
strategy of the first player such that, for all $\eps > 0$,
there exists an integer $T_\eps$ such that for all strategies of the second player,
\[
\P \left\{ \forall \, T \geq T_\eps, \quad \inf_{c \in \cC} \ \norm[c - \frac{1}{T} \sum_{t=1}^T m \bigl( \bx_t,\by_t \bigr)]_2 \leq \eps \right\}
\ \geq 1-\eps\,.
\]
\end{definition}

As indicated below, in this setting the first player may even have deterministic strategies such that,
for all (deterministic or randomized) strategies of the second player,
\[
\inf_{c \in \cC} \ \norm[c - \frac{1}{T} \sum_{t=1}^T m \bigl( \bx_t,\by_t \bigr)]_2 \leq \beta(T)
\]
with probability $1$, where $\beta(T) \to 0$.

\paragraph{Necessary and sufficient condition for approachability.} For
closed convex sets there is a simple characterization of approachability that is a
direct consequence of the minimax theorem; the condition is the same for the two settings,
whether pure or mixed actions are taken and observed.

\begin{theorem}[Theorem~3 of \citemor{Bla56}]
\label{th:appr}
A closed convex set $\cC \subseteq \R^d$ is approachable (with pure or mixed actions) if and only if
\[
\forall \, \by \in \Delta(\cB), \ \ \exists \, \bx \in \Delta(\cA), \qquad \quad
m(\bx,\by) \in \cC\,.
\]
\end{theorem}

\paragraph{An associated strategy (that is efficient depending
on the geometry of $\cC$).} Blackwell
suggested a simple strategy with a geometric flavor; it only
requires a bandit monitoring.

Play an arbitrary $\bx_1$. For $t \geq 1$, given the vector-valued quantities
\[
\hat{m}_t = \frac{1}{t} \sum_{s=1}^t m(A_s,B_s)
\qquad \mbox{or} \qquad
\hat{m}_t = \frac{1}{t} \sum_{s=1}^t m(\bx_s,\by_s)\,,
\]
depending on whether pure or mixed actions are taken and observed,
compute the projection $c_t$ (in $\ell^2$--norm) of $\hat{m}_t$ on $\cC$.
Find a mixed action $\bx_{t+1}$ that solves the minimax equation
\begin{equation}
\label{eq:minmax}
\min_{\bx \in \Delta(\cA)} \, \max_{\by \in \Delta(\cB)} \,\, \bigl\langle \hat{m}_t - c_t , \, m(\bx,\by) \bigr\rangle\,,
\end{equation}
where $\langle \,\cdot\,, \,\cdot\, \rangle$ is the Euclidian inner product in $\R^d$.
In the case when pure actions are taken and observed, draw $A_{t+1}$ at random according to $\bx_{t+1}$.

The minimax problem used above to determine $\bx_{t+1}$ is easily seen to
be a (scalar) zero-sum game and is therefore efficiently solvable using, e.g., linear programming: the associated
complexity is polynomial in $N_\cA$ and $N_\cB$. All in all, this strategy is efficient if the computations of
the required projections onto $\cC$ in $\ell^2$--norm can be performed efficiently.

The strategy presented above enjoys the following rates of convergence
for approachability.

\begin{theorem}[Theorem~3 of \citemor{Bla56}; Theorem~II.4.3 of \citemor{MeSoZa94}]
\label{th:Bla}
We denote by $M$ a bound in norm over $m$, i.e.,
\[
\max_{(a,b) \in \cA \times \cB} \bnorm[m(a,b)]_2 \leq M\,.
\]
With mixed actions taken and observed, the above strategy ensures that
for all strategies of the second player, with probability 1,
\[
\inf_{c \in \cC} \ \norm[c - \frac{1}{T} \sum_{t=1}^T m \bigl( \bx_t, \by_t \bigr)]_2 \ \leq \frac{2M}{\sqrt{T}}\,;
\]
while with pure actions taken and observed, for all $\delta \in (0,1)$ and for all
strategies of the second player, with probability at least $1-\delta$,
\[
\sup_{\tau \geq T} \ \inf_{c \in \cC} \ \norm[c - \frac{1}{\tau} \sum_{t=1}^\tau m \bigl( A_t,B_t \bigr)]_2 \ \leq
2M \sqrt{\frac{2}{\delta T}}\,.
\]
\end{theorem}

\paragraph{An alternative strategy in the case
where pure actions are taken and observed.}

Convergence rates of a slightly different flavor (but still implying approachability)
can be proved, in the full monitoring case, by modifying the above procedure as follows.
For $t \geq 1$, consider instead the vector-valued quantity
\[
\hat{m}_t = \frac{1}{t} \sum_{s=1}^t m(\bx_s,B_s)\,,
\]
compute its projection $c_t$ (in $\ell^2$--norm) on $\cC$,
and solve the associated minimax problem~\eqref{eq:minmax}.

This modified strategy enjoys the following rates of convergence
for approachability when pure actions are taken and observed.

\begin{theorem}[Section~7.7 and Exercise~7.23 of~\citemor{CeLuSt06}]
\label{th:CBL}
We denote by $M$ a bound in norm over $m$, i.e.,
\[
\max_{(a,b) \in \cA \times \cB} \bnorm[m(a,b)]_2 \leq M\,.
\]
With pure actions taken and observed, the above strategy ensures that
for all strategies of the second player, with probability at least $1-\delta$,
\[
\inf_{c \in \cC} \ \norm[c - \frac{1}{T} \sum_{t=1}^T m \bigl( A_t,B_t \bigr)]_2 \ \leq
\frac{2M}{\sqrt{T}} \Bigl( 1 + 2\sqrt{\ln (2/\delta)} \Bigr)\,.
\]
\end{theorem}

In the next section, we will rather resort to this slightly modified procedure
as the form of the resulting bounds is closer to the one derived in
the main section (Section~\ref{sec:main}) of this paper.

\section{Robust approachability for finite set-valued games.}
\label{se:robapproach}

In this section we extend the results from the previous section to set-valued payoff functions
in the case of full monitoring.
We denote by $\cS$ the set of all subsets of $\R^d$ and consider a set-valued
payoff function $\om : \cA \times \cB \to \cS$.

\paragraph{Pure actions taken and observed.} At each round $t$, the players
choose simultaneously respective actions $A_t \in \cA$ and $B_t \in \cB$,
possibly at random according to mixed distributions $\bx_t$ and $\by_t$.
Full monitoring takes place for the first player: he observes $B_t$
at the end of round $t$.
However, as a result, the first player gets the {\em subset} $\om(A_t,B_t)$ as a payoff. This models the
ambiguity or uncertainty associated with some true underlying payoff gained.

We extend $\om$ multi-linearly to $\Delta(\cA) \times \Delta(\cB)$ and even
to $\Delta(\cA \times \cB)$, the set of joint probability distributions
on $\cA \times \cB$, as follows. Let
\[
\mu = \bigl( \mu_{a,b} \bigr)_{(a,b) \in \cA \times \cB}
\]
be such a joint probability distribution; then $\om(\mu)$ is defined
as a finite convex combination\footnote{For
two sets $S,\,\,T$ and $\alpha \in [0,1]$, the convex combination $\alpha S + (1-\alpha)T$
is defined as
\begin{center} $
\bigl\{\alpha s + (1-\alpha)t, \ \ s \in S \ \mbox{and} \ t \in T \bigr\}\,.
$ \end{center}
}
of subsets of $\R^d$,
\[
\om(\mu) = \sum_{a \in \cA} \sum_{b \in \cB} \mu_{a,b} \, \om(a,b)\,.
\]
When $\mu$ is the product-distribution of some $\bx \in \Delta(\cA)$ and
$\by \in \Delta(\cB)$, we use the notation $\om(\mu) = \om(\bx,\by)$.

We denote by
\[
\pi_T = \frac{1}{T} \sum_{t=1}^T \delta_{(A_t,B_t)}
\]
the empirical distribution of the pairs $(A_t,B_t)$ of actions taken
during the first $T$ rounds, and will be interested in the behavior of
\[
\frac{1}{T} \sum_{t=1}^T \om(A_t,B_t)\,,
\]
which can also be rewritten here in a compact way as $\om(\pi_T)$, by linearity of the extension of $\om$.

The distance of this set $\om(\pi_T)$ to the target set $\cC$ will be measured in a worst-case sense:
we denote by
\[
\epsilon_T = \sup_{d \in \om(\pi_T)} \, \inf_{c \in \cC} \ \norm[c - d]_2
\]
the smallest value such that $\om(\pi_T)$ is included in an $\epsilon_T$--neighborhood of $\cC$.
Robust approachability of a set $\cC$ with the set-valued payoff function $\om$
then simply means that the sequence of $\epsilon_T$ tends almost-surely to $0$, uniformly
with respect to the strategies of the second player.

\begin{definition}
A set $\cC \subseteq \R^d$ is $\om$--\emph{robust approachable with pure actions} if there exists a strategy
of the first player such that, for all $\eps > 0$,
there exists an integer $T_\eps$ such that for all strategies of the second player,
\[
\P \left\{ \forall \, T \geq T_\eps, \quad \sup_{d \in \om(\pi_T)} \, \inf_{c \in \cC} \ \norm[c - d]_2
\leq \eps \right\}
\ \geq 1-\eps\,.
\]
\end{definition}

\paragraph{Mixed actions taken and observed.} At each round $t$, the players
choose simultaneously respective mixed actions $\bx_t \in \Delta(\cA)$ and $\by_t \in \Delta(\cB)$.
Full monitoring still takes place for the first player: he observes $\by_t$
at the end of round $t$; he however
gets the subset $\om(\bx_t,\by_t)$ as a payoff (which, again, accounts for the uncertainty).

The product-distribution of two elements $\bx = (x_a)_{a \in \cA} \in \Delta(\cA)$ and $\by = (y_b)_{b \in \cB} \in \Delta(\cB)$
will be denoted by $\bx \otimes \by$; it gives a probability mass of $x_a y_b$ to each pair $(a,b) \in \cA \times \cB$.
We consider the empirical joint distribution of mixed actions taken
during the first $T$ rounds,
\[
\nu_T = \frac{1}{T} \sum_{t=1}^T \bx_t \otimes \by_t\,,
\]
and will be interested in the behavior of
\[
\frac{1}{T} \sum_{t=1}^T \om(\bx_t,\by_t)\,,
\]
which can also be rewritten here in a compact way as $\om(\nu_T)$, by linearity of the extension of $\om$.

\begin{definition}
A set $\cC \subseteq \R^d$ is $\om$--\emph{robust approachable with mixed actions} if there exists a strategy
of the first player such that, for all $\eps > 0$,
there exists an integer $T_\eps$ such that for all strategies of the second player,
\[
\P \left\{ \forall \, T \geq T_\eps, \quad \sup_{d \in \om(\nu_T)} \, \inf_{c \in \cC} \ \norm[c - d]_2
\leq \eps \right\}
\ \geq 1-\eps\,.
\]
\end{definition}

Actually, the bounds exhibited below in this setting will be of the form
\[
\sup_{d \in \om(\nu_T)} \, \inf_{c \in \cC} \ \norm[c - d]_2 \leq \beta(T)
\]
with probability $1$ and uniformly over all (deterministic or randomized)
strategies of the second player, where $\beta(T) \to 0$ and for deterministic strategies
of the first player.

\paragraph{A useful continuity lemma.} Before proceeding we provide a continuity
lemma. It can be reformulated as indicating that for all joint distributions
$\mu$ and $\nu$ over $\cA \times \cB$, the set $\om(\mu)$ is contained
in a $M \norm[\mu - \nu]_1$--neighborhood of $\om(\nu)$,
where $M$ is a bound in $\ell^2$--norm on $\om$; this is
a fact that we will use repeatedly below.

\begin{lemma}
\label{lm:DVT}
Let $\mu$ and $\nu$ be two probability distributions over $\cA \times \cB$.
We assume that the set-valued function $\om$ is bounded in norm by $M$,
i.e., that there exists a real number $M > 0$ such that
\[
\forall (a,b) \in \cA \times \cB, \qquad \quad \sup_{d \in \om(a,b)} \norm[d]_2 \leq M\,.
\]
Then
\[
\sup_{d \in \om(\mu)} \,\, \inf_{c \in \om(\nu)} \,\, \norm[d - c]_2 \,\,
\leq M \norm[\mu - \nu]_1 \leq M\sqrt{N_{\cA} N_{\cB}} \norm[\mu - \nu]_2\,,
\]
where the norms in the right-hand side are respectively the $\ell^1$ and $\ell^2$--norms between probability
distributions.
\end{lemma}

\begin{proof}
Let $d$ be an element of $\om(\mu)$; it can be written as
\[
d = \sum_{a \in \cA} \sum_{b \in \cB} \mu_{a,b} \, \theta_{a,b}
\]
for some elements $\theta_{a,b} \in \om(a,b)$. We consider
\[
c = \sum_{a \in \cA} \sum_{b \in \cB} \nu_{a,b} \, \theta_{a,b}\,,
\]
which is an element of $\om(\nu)$. Then by the triangle inequality,
\[
\norm[d - c]_2 = \norm[ \sum_{a \in \cA} \sum_{b \in \cB} \, \bigl( \mu_{a,b} - \nu_{a,b} \bigr) \theta_{a,b} ]_2
\leq \sum_{a \in \cA} \sum_{b \in \cB} \, \bigl| \mu_{a,b} - \nu_{a,b} \bigr| \norm[ \theta_{a,b} ]_2
\leq M \,\, \sum_{a \in \cA} \sum_{b \in \cB} \, \bigl| \mu_{a,b} - \nu_{a,b} \bigr|\,.
\]
This entails the first claimed inequality. The second one follows from an application of the
Cauchy-Schwarz inequality.
\end{proof}

\begin{corollary}
\label{cor:contD}
When the set-valued function $\om$ is bounded in norm, for all $\by \in \Delta(\cB)$,
the mapping $D_{\by} : \Delta(A) \to \mathbb{R}$ defined by
\[
\forall\, \bx \in \Delta(\cA), \qquad D_{\by}(\bx) = \sup_{d \in \om(\bx,\by)} \inf_{c \in \cC} \norm[c-d]_2
\]
is continuous.
\end{corollary}

\begin{proof}
We show that for all $\bx,\,\bx' \in \Delta(\cA)$,
the condition $\norm[\bx'-\bx]_1 \leq \varepsilon$ implies that $D_{\by}(\bx) - D_{\by}(\bx') \leq M\e$,
where $M$ is the bound in norm over $\om$.
Indeed, fix $\delta > 0$ and let $d_{\delta,\bx} \in \om(\bx,\by)$ be such that
\begin{equation}
\label{eq:Ddelta}
D_{\by}(\bx)
\leq \inf_{c \in \cC}
\bnorm[ {c-d_{\delta,\bx}} ]_2
+ \delta\,.
\end{equation}
By Lemma \ref{lm:DVT} (with the choices $\mu = \bx \otimes \by$ and $\nu = \bx' \otimes \by$)
there exists $d_{\delta,\bx'} \in \om(\bx',\by)$ such that
$\bnorm[d_{\delta,\bx} - d_{\delta,\bx'}]_2 \leq M\e + \delta$. The triangle inequality entails that
\[
\inf_{c \in \cC} \bnorm[c-d_{\delta,\bx}]_2 \leq \inf_{c \in \cC} \bnorm[c-d_{\delta,\bx'}]_2
+ M\e + \delta\,.
\]
Substituting in~(\ref{eq:Ddelta}), we get that
\[
D_{\by}(\bx) \leq M\e + 2\delta + \inf_{c \in \cC} \bnorm[c-d_{\delta,\bx'}]_2 \leq
M\e + 2\delta + D_{\by}(\bx')\,,
\]
which, letting $\delta \to 0$, proves our continuity claim.
\end{proof}

\paragraph{Necessary and sufficient condition for robust approachability.}
This conditions reads as follows and will be referred to as~\eqref{eq:RAC},
an acronym that stands for ``robust approachability condition.''

\begin{theorem}
\label{th:rob-appr}
Suppose that the set-valued function $\om$ is bounded in norm by $M$.
A closed convex set $\cC \subseteq \R^d$ is $\om$--approachable (with pure or mixed actions) if and only if
the following robust approachability condition is satisfied,
\begin{equation}
\label{eq:RAC}
\tag{\RAC}
\forall \, \by \in \Delta(\cB), \ \ \exists \, \bx \in \Delta(\cA), \qquad \quad
\om(\bx,\by) \subseteq \cC\,.
\end{equation}
\end{theorem}

\begin{proofof}{the necessity of Condition~\eqref{eq:RAC}}
If the condition does not hold, then there exists $\by_0 \in \Delta(\cB)$
such that for every $\bx \in \cA$, the set $\om(\bx,\by_0)$ is
not included in $\cC$, i.e., it contains at least one point not in $\cC$.
We consider the mapping $D_{\by_0}$ defined in the statement of Corollary~\ref{cor:contD}.
Since $\cC$ is closed, distances of given individual points to $\cC$ are achieved; therefore,
by the choice of $\by_0$, we get that $D_{\by_0}(\bx) > 0$ for all $\bx \in \Delta(\cA)$.
Now, since $D_{\by_0}$ is continuous on the compact set $\Delta(\cA)$, as
asserted by the indicated corollary, it attains its minimum, whose value we denote by $D_{\min} > 0$.

Assume now that the second player chooses at each round $\by_t = \by_0$ as his mixed action.
In the case of mixed actions taken and observed, denoting
\[
\ol{\bx}_T = \frac{1}{T} \sum_{t=1}^T \bx_t\,,
\]
we get that $\nu_t = \ol{\bx}_T \otimes \by_0$, and hence, for all
strategies of the first player and for all $T \geq 1$,
\[
\sup_{d \in \om(\nu_T)} \inf_{c \in \cC} \norm[c-d]_2 = D_{\by_0}(\ol{\bx}_T) \geq D_{\min} > 0\,,
\]
which shows that $\cC$ is not approachable.

The case of pure actions taken and observed is treated similarly,
with the sole addition of a concentration argument. By martingale convergence (e.g.,
repeated uses of the Hoeffding-Azuma inequality together with an application
of the Borel-Cantelli lemma),
$\delta_T = \norm[\pi_T - \nu_T]_1 \to 0$
almost surely as $T \to \infty$. By applying Lemma~\ref{lm:DVT}, we get
\[
\sup_{d \in \om(\pi_T)} \inf_{c \in \cC} \norm[c-d]_2 \geq
\sup_{d \in \om(\nu_T)} \inf_{c \in \cC} \norm[c-d]_2 - M\delta_T
\geq D_{\min} - M\delta_T\
\]
and simply take the $\liminf$ in the above inequalities to conclude the argument.
\end{proofof}

That~\eqref{eq:RAC} is sufficient to get robust approachability is proved in a constructive
way, by exhibiting suitable strategies. We identify probability distributions over $\cA \times \cB$ with
vectors in $\R^{\cA \times \cB}$ and consider the vector-valued payoff function
\[
m : (a,b) \in \cA \times \cB \longmapsto \delta_{(a,b)} \in \R^{\cA \times \cB}\,,
\]
which we extend multi-linearly to $\Delta(\cA) \times \Delta(\cB)$; the target
set will be
\begin{equation}
\label{eq:defctilde}
\widetilde{\cC} = \bigl\{ \mu \in \Delta ( \cA \times \cB ) : \ \ \om(\mu) \subseteq \cC \bigr\}\,.
\end{equation}
Since $\om$ is a linear function on $\Delta ( \cA \times \cB )$ and $\cC$ is convex,
the set $\widetilde{\cC}$ is convex as well. In addition, since $\cC$ is closed,
$\widetilde{\cC}$ is also closed.

\begin{lemma}
\label{lm:eqvRAC}
Condition~\eqref{eq:RAC} is equivalent to the $m$--approachability of $\widetilde{\cC}$.
\end{lemma}

\begin{proof}
This equivalence is immediate via Theorem~\ref{th:appr}.
The latter indeed states that the $m$--approachability of $\widetilde{\cC}$
is equivalent to the fact that for all $\by \in \Delta(\cB)$, there exists some $\bx \in \Delta(\cA)$
such that $\mu = m(\bx,\by)$, the product-distribution between $\bx$ and $\by$, belongs to
$\widetilde{\cC}$, i.e., satisfies $\om(\mu) = \om(\bx,\by) \subseteq \cC$.
\end{proof}

The above definition of $m$ entails the following rewriting,
\[
\pi_T = \frac{1}{T} \sum_{t=1}^T m(A_t,B_t) \qquad \mbox{and} \qquad
\nu_T = \frac{1}{T} \sum_{t=1}^T m(\bx_t,\by_t)\,.
\]
Let $P_{\widetilde{\cC}}$ denote the projection operator onto $\widetilde{\cC}$;
the quantities at hand in the definition of $m$--approachability of $\widetilde{\cC}$ are given by
\[
\e_T = \Bnorm[ \pi_T - P_{\widetilde{\cC}}(\pi_T) ]_2 = \inf_{\mu \in \widetilde{\cC}} \norm[\pi_T - \mu]_2
\qquad \mbox{and} \qquad
\e'_T = \Bnorm[ \nu_T - P_{\widetilde{\cC}}(\nu_T) ]_2 = \inf_{\mu \in \widetilde{\cC}} \norm[\nu_T - \mu]_2\,.
\]
We now relate the quantities of interest, i.e., the ones arising in the definition of $\om$--robust
approachability of $\cC$, to the former quantities.

\begin{lemma}
\label{lm:apprtorobappr}
With pure actions taken and observed,
\[
\sup_{d \in \om(\pi_T)} \, \inf_{c \in \cC} \ \norm[c - d]_2 \leq M \sqrt{N_{\cA} N_{\cB}} \,\, \e_T\,.
\]
With mixed actions taken and observed,
\[
\sup_{d \in \om(\nu_T)} \, \inf_{c \in \cC} \ \norm[c - d]_2 \leq M \sqrt{N_{\cA} N_{\cB}} \,\, \e'_T\,.
\]
\end{lemma}

\begin{proof}
Lemma~\ref{lm:DVT} entails that
the sets $\om(\pi_T)$ are included in $M\sqrt{N_{\cA} N_{\cB}}\,\epsilon_T$--neighborhoods of
$\om \bigl( P_{\widetilde{\cC}}(\pi_T) \bigr)$. Since by definition of $\widetilde{\cC}$,
one has $\om \bigl( P_{\widetilde{\cC}}(\pi_T) \bigr) \subseteq \cC$, we get in particular that
the sets $\om(\pi_T)$ are included in
$M\sqrt{N_{\cA} N_{\cB}}\,\epsilon_T$--{neigh\-bor\-hoods} of $\cC$, which is exactly what was stated.
The argument can be repeated with the $\nu_T$ to get the second bound in the statement of the lemma.
\end{proof}

\begin{proofof}{the sufficiency of Condition~\eqref{eq:RAC}}
First, Lemma~\ref{lm:eqvRAC} shows that
Condition~\eqref{eq:RAC} (via Theorems~\ref{th:Bla} or~\ref{th:CBL})
ensures the existence of strategies $m$--approaching $\widetilde{\cC}$.
Second, Lemma~\ref{lm:apprtorobappr} indicates that these strategies
also $\om$--robust approach $\cC$. (It even translates the rates for the
$m$--approachability of $\widetilde{\cC}$
into rates for the $\om$--robust approachability of $\cC$;
for instance, in the case of mixed actions taken and observed,
the $2/\sqrt{T}$ rate for the $m$--approachability of $\widetilde{\cC}$
becomes a $2M \sqrt{N_{\cA} N_{\cB} / T}$ rate for the $\om$--robust approachability of $\cC$,
a fact that we will use in the proof of Theorem~\ref{TheoGeneralCase}.)
\end{proofof}

\paragraph{Two concluding remarks.}
Note that, as explained around Equation (\ref{eq:minmax}),
the considered strategies for $m$--approaching $\widetilde{\cC}$,
or equivalently  $\om$--robust approaching $\cC$, are efficient as soon as projections
in $\ell^2$--norm onto the set $\widetilde{\cC}$ defined in~(\ref{eq:defctilde})
can be computed efficiently. The latter fact depends on the respective geometries
of $\om$ and $\cC$. We will provide examples of favorable cases (see, e.g., Section~\ref{sec:ER}
about minimization of external regret under partial monitoring).

A final remark is that the proposed strategies require full monitoring, as they
rely on the observations of either the pair of played mixed actions $m(\bx_t,\by_t)$
or of played pure actions $m(A_t,B_t)$. They enjoy no obvious extension to a case
where only a bandit monitoring of the played sets $\om(\bx_t,\by_t)$ or $\om(A_t,B_t)$
would be available.

\section{Robust approachability for concave--convex set-valued games.}
\label{sec:robappr-gal}

We consider in this section the same setting of mixed actions taken and observed
as in the previous section, that is, we deal with set-valued payoff functions
$\om : \Delta(\cA) \times \Delta(\cB) \to \cS$ under full monitoring.
However, in the previous section $\om$ was linear on
$\Delta(\cA) \times \Delta(\cB)$, an assumption that we now weaken
while still having that~\eqref{eq:RAC} is the necessary and sufficient condition
for robust approachability. The price to pay for this is
the loss of the possible efficiency of the approachability strategies exhibited
and the worsening of the convergence rates.

Formally, the functions $\om : \Delta(\cA) \times \Delta(\cB) \to \cS$ that we will consider
will satisfy one or several of the following properties.

\begin{definition}
\label{def:unifcont}
A function $\om : \Delta(\cA) \times \Delta(\cB) \to \cS$ is uniformly continuous in
its first argument if for all $\e > 0$, there exists $\eta > 0$ such that
for all $\bx,\bx' \in \Delta(\cA)$
satisfying $\norm[\bx-\bx']_1 \leq \eta$ and for all $\by \in \Delta(\cB)$, the
set $\om(\bx',\by)$ is included in an $\e$--neighborhood of $\om(\bx,\by)$
in the Euclidian norm. Put differently,
\[
\sup_{d \in \om(\bx',\by)} \,\, \inf_{c \in \om(\bx,\by)} \,\, \norm[d - c]_2 \, \leq \e
\qquad \mbox{or} \qquad \om(\bx',\by) \subseteq \om(\bx,\by) + \e \bB\,,
\]
where $\bB$ is the unit Euclidian ball in $\R^d$.
\end{definition}

Uniform continuity in the second argument is defined symmetrically.

\begin{definition}
\label{def:convconc}
A function $\om : \Delta(\cA) \times \Delta(\cB) \to \cS$ is
concave in its first argument if for all $\bx,\bx' \in \Delta(\cA)$,
all $\by \in \Delta(\cB)$, and all $\alpha \in [0,1]$,
\[
\om\bigl( \alpha \bx + (1-\alpha) \bx', \, \by \bigr)
\subseteq \alpha \, \om(\bx,\by) + (1-\alpha) \, \om(\bx',\by)\,.
\]
A function $\om : \Delta(\cA) \times \Delta(\cB) \to \cS$ is
convex in its second argument if for all $\bx \in \Delta(\cA)$,
all $\by,\by' \in \Delta(\cB)$, and all $\alpha \in [0,1]$,
\[
\alpha \, \om(\bx,\by) + (1-\alpha) \, \om(\bx,\by') \subseteq
\om\bigl( \bx, \, \alpha \by + (1-\alpha) \by' \bigr)\,.
\]
\end{definition}

An example of such a function $\om$ is discussed in Lemma~\ref{lm:cvxconc}.

The following theorem indicates that~\eqref{eq:RAC} is the necessary and sufficient condition for
the $\om$--robust approachability of a closed convex set $\cC$ with mixed actions
when the payoff function $\om$ satisfies all four properties stated above.
(Boundedness of $\om$ indeed follows from the continuity of $\om$ in each variable.)

\begin{theorem}
\label{th:robappr-gal}
If $\om$ is bounded, convex, and uniformly continuous in its second argument,
then~\eqref{eq:RAC} entails that a closed convex set $\cC$ is $\om$--robust approachable
with mixed actions. \\
\indent On the contrary, if $\om$ is concave and uniformly continuous in its first argument,
then a closed convex set $\cC$ can be $\om$--robust approachable with mixed actions
only if~\eqref{eq:RAC} is satisfied.
\end{theorem}

\begin{proofof}{the second statement of Theorem~\ref{th:robappr-gal}}
The proof of Corollary~\ref{cor:contD} extends to the case considered here
and shows, thanks to the ad hoc consideration of the result stated in Lemma~\ref{lm:DVT}
as following from Definition~\ref{def:unifcont}, that for all $\by \in \Delta(\cB)$, the mapping
$D_{\by}$ is still continuous over $\Delta(\cA)$.
We now proceed by contradiction and assume that~\eqref{eq:RAC}
is not satisfied; the first part of the proof of the necessity of~\eqref{eq:RAC}
in Theorem~\ref{th:rob-appr} also applies to the present case: there exists
$\by_0$ such that $D_{\by_0} \geq D_{\min} > 0$ over $\Delta(\cA)$.
It then suffices to note that whenever the second player resorts to $\by_t = \by_0$
at all rounds $t \geq 1$, then for all strategies of the first player,
the quantity of interest in robust approachability
can be lower bounded as follows, thanks to the concavity in the first argument:
\begin{multline}
\nonumber
\sup \left\{ \inf_{c \in \cC} \,\, \norm[d - c]_2 : \ \ d \in \frac{1}{T} \sum_{t=1}^T \om(\bx_t,\by_0) \right\} \\
\geq
\sup \left\{ \inf_{c \in \cC} \,\, \norm[d - c]_2 : \ \ d \in \om\!\left(\frac{1}{T} \sum_{t=1}^T \bx_t, \,\, \by_0\right) \right\}
= D_{\by_0}\!\left(\frac{1}{T} \sum_{t=1}^T \bx_t\right) \geq D_{\min} > 0\,.
\end{multline}
Therefore, $\cC$ is $\om$--robust approachable with mixed actions by no strategy of the first player.
\end{proofof}

The proof of the first statement of Theorem~\ref{th:robappr-gal}
relies on the use of approximately calibrated strategies of the first player,
as introduced and studied (among others) by~\citemor{Dawid82}, \citemor{FoVo98}, \citemor{MaSt10}.
Formally, given $\eta > 0$, an $\eta$--calibrated strategy of the first player considers some
finite covering of $\Delta(\cB)$ by $N_{\eta}$ balls of radius ${\eta}$ and abides by the following constraints.
Denoting by $\by^1,\ldots,\by^{N_{\eta}}$ the centers of the balls in the covering (they form what
will be referred to later on as an $\eta$--grid),
such a strategy chooses only forecasts in $\bigl\{ \by^1,\ldots,\by^{N_{\eta}} \bigr\}$.
We thus denote by $L_t$ the index chosen in $\bigl\{ 1,\ldots,N_{\eta} \bigr\}$ at round $t$
and by
\[
N_T(\ell) = \sum_{t=1}^T \ind_{ \{ L_t = \ell \} }
\]
the total number of rounds within the first $T$ ones when the element $\ell$ of
the grid was chosen.
We denote by $(\,\cdot\,)_+$ the function that gives the nonnegative
part of a real number.
The final condition to be satisfied is that for all $\delta > 0$, there exists
an integer $T_\delta$ such that for all strategies of the second player,
with probability at least $1 - \delta$,
for all $T \geq T_\delta$,
\begin{equation}
\label{def:etacal}
\sum_{\ell = 1}^{N_\eta} \frac{N_T(\ell)}{T} \left( \norm[\by^\ell - \frac{1}{N_T(\ell)} \sum_{t=1}^T \by_t \ind_{ \{ L_t = \ell \} }
]_1 - \eta \right)_{\!\! +} \,\, \leq \delta\,.
\end{equation}
This calibration criterion is slightly stronger than the classical $\eta$--calibration score
usually considered in the literature, which consists of omitting nonnegative parts in the
criterion above and ensuring that for all strategies of the second player, with probability at least $1 - \delta$,
for all $T \geq T_\delta$,
\begin{equation}
\label{def:etacal2}
\sum_{\ell = 1}^{N_\eta} \frac{N_T(\ell)}{T} \, \norm[\by^\ell - \frac{1}{N_T(\ell)} \sum_{t=1}^T \by_t \ind_{ \{ L_t = \ell \} }
]_1 \,\, \leq \eta + \delta\,.
\end{equation}
The existence of a calibrated strategy in the sense of~\eqref{def:etacal}
however follows from the same approachability-based construction studied in~\citemor{MaSt10}
to get~\eqref{def:etacal2} and is detailed in the appendix. In the sequel we will only use
the following consequence of calibration: that
for all strategies of the second player, with probability at least $1 - \delta$,
for all $T \geq T_\delta$,
\begin{equation}
\label{def:etacal3}
\max_{\ell = 1,\ldots,N_\eta} \frac{N_T(\ell)}{T} \left( \norm[\by^\ell - \frac{1}{N_T(\ell)} \sum_{t=1}^T \by_t \ind_{ \{ L_t = \ell \} }
]_1 - \eta \right)_{\!\! +} \,\, \leq \delta\,.
\end{equation}

\begin{proofof}{the first statement of Theorem~\ref{th:robappr-gal}}
The insight of this proof is similar to the one illustrated in~\citemor{Per09}.
We first note that it suffices to prove that for all $\epsilon > 0$, the set $\cC_\epsilon$ defined as
the $\epsilon$--neighborhood of $\cC$
is $\om$--robust approachable with mixed actions; this is so up to
proceeding in regimes $r = 1,\,2,\,\ldots$
each corresponding to a dyadic value $\epsilon_r = 2^{-r}$ and lasting for a number of rounds
carefully chosen in terms of the length of the previous regimes.

Therefore, we fix $\e > 0$ and associate with it a modulus of continuity $\eta > 0$ given by the
uniform continuity of $\om$ in its second argument.
We consider an $\eta/2$--calibrated strategy of the first player, which we will use as an auxiliary
strategy. Since~\eqref{eq:RAC} is satisfied, we may associate with each element $\by^\ell$ of
the underlying $\eta/2$--grid a mixed action $\bx^\ell \in \Delta(\cA)$ such that
$\om\bigl( \bx^\ell, \by^\ell \bigr) \subseteq \cC$. The main strategy
of the first player then prescribes the use of $\bx_t = \bx^{L_t}$ at each round $t \geq 1$.
The intuition behind this definition is that if $\by^{L_t}$ is forecast by the auxiliary strategy,
then since the latter is calibrated, one should play as good as possible against $\by^{L_t}$; in view of the
aim at hand, which is approaching $\cC$, such a good reply is given by $\bx^{L_t}$.

To assess the constructed strategy, we group rounds according to the values $\ell$ taken
by the $L_t$; to that end, we recall that
$N_T(\ell)$ denotes the number of rounds in which $\by^\ell$ was forecast and $\bx^\ell$ was played.
The average payoff up to round $T$ is then rewritten as
\[
\frac{1}{T} \sum_{t=1}^T \om(\bx_t,\by_t) = \sum_{\ell=1}^{N_{\eta/2}} \frac{N_T(\ell)}{T} \left( \frac{1}{N_T(\ell)} \sum_{t=1}^T
\om\bigl(\bx^\ell,\by_t\bigr) \ind_{ \{ L_t = \ell \} } \right)\,.
\]
We denote for all $\ell$ such that $N_T(\ell) > 0$
the average of their corresponding mixed actions $\by_t$ by
\[
\ol{\by}_T^\ell = \frac{1}{N_T(\ell)} \sum_{t=1}^T
\by_t \ind_{ \{ L_t = \ell \} }\,.
\]
The convexity of $\om$ in its second argument leads to the
inclusion
\[
\frac{1}{T} \sum_{t=1}^T \om(\bx_t,\by_t) =
\sum_{\ell=1}^{N_{\eta/2}} \frac{N_T(\ell)}{T} \left( \frac{1}{N_T(\ell)} \sum_{t=1}^T
\om\bigl(\bx^\ell,\by_t\bigr) \ind_{ \{ L_t = \ell \} } \right)
\subseteq
\sum_{\ell=1}^{N_{\eta/2}} \frac{N_T(\ell)}{T} \, \om\bigl(\bx^\ell,\ol{\by}_T^\ell\bigr)\,.
\]
To show that the above-defined strategy $\om$--robust approaches $\cC_\epsilon = \cC + \e \bB$,
it suffices to show that for all $\delta > 0$,
there exists an integer $T'_\delta$ such that for all strategies of the second player,
\[
\P \! \left\{ \forall \, T \geq T'_\delta, \quad
\sum_{\ell=1}^{N_{\eta/2}} \frac{N_T(\ell)}{T} \, \om\bigl(\bx^\ell,\ol{\by}_T^\ell\bigr)
\subseteq \cC + (\e+\delta)\bB \right\} \geq 1-\delta\,.
\]

We denote by $M$ a bound in $\ell^2$--norm on $\om$, i.e.,
for all $\bx \in \Delta(\cA)$ and $\by \in \Delta(\cB)$, the inclusion
$\om(\bx,\by) \subseteq M \bB$ holds.
We let $\delta' = \delta (\eta/2) \big/ \bigl(M \, N_{\eta/2} \bigr)$
and define $T'_\delta$ as the time $T_{\delta'}$
corresponding to~\eqref{def:etacal3}. All statements that follow
will be for all strategies of the second player and with probability at least $1 - \delta' \geq
1 - \delta$, for all $T \geq T'_\delta$, as required. For each index $\ell$ of the grid,
either $\delta' T / N_T(\ell) \leq \eta/2$ or
$\delta' T / N_T(\ell) > \eta/2$.
In the first case, following~\eqref{def:etacal3},
$\norm[\by^\ell - \ol{\by}_T^\ell] \leq \eta/2 + \delta' T / N_T(\ell) \leq \eta$;
since $\eta$ is the modulus of continuity for $\e$,
we get that
\[
\frac{N_T(\ell)}{T} \, \om\bigl(\bx^\ell,\ol{\by}_T^\ell\bigr) \subseteq
\frac{N_T(\ell)}{T} \left( \om\bigl(\bx^\ell,\by^\ell\bigr) + \e \bB \right)
\subseteq
\frac{N_T(\ell)}{T} \bigl( \cC + \e \bB \bigr)\,,
\]
where we used the definition of $\bx^\ell$ to get the second inclusion.
In the second case, using the boundedness of $\om$, we simply write
\[
\frac{N_T(\ell)}{T} \, \om\bigl(\bx^\ell,\ol{\by}_T^\ell\bigr) \subseteq
\frac{N_T(\ell)}{T} \, M \bB \subseteq \frac{\delta'}{\eta/2} \, M \bB\,.
\]
Summing these bounds over $\ell$ yields
\[
\sum_{\ell=1}^{N_{\eta/2}} \frac{N_T(\ell)}{T} \, \om\bigl(\bx^\ell,\ol{\by}_T^\ell\bigr)
\subseteq \cC + \e\bB + \frac{N_{\eta/2} \delta'}{\eta/2} \, M \, \bB
= \cC + (\e+\delta)\,\bB\,,
\]
where we used the definition of $\delta'$ in terms of $\delta$. This concludes the proof.
\end{proofof}


\section{Approachability in games with partial monitoring: statement of the necessary and sufficient
condition; links with robust approachability.}
\label{se:apptogames}

A repeated vector-valued game with partial monitoring is described as follows
(see, e.g., \citemor{MeSoZa94}, \citemor{Rus99}, and the references therein).
The players have respective finite action sets $\cI$ and $\cJ$.
We denote by $r : \cI \times \cJ \to \R^d$ the vector-valued payoff function of the first player
and extend it multi-linearly to $\Delta(\cI) \times \Delta(\cJ)$.
At each round, players simultaneously choose their actions $I_t \in \cI$ and $J_t \in \cJ$,
possibly at random according to probability distributions denoted by
$\bp_t \in \Delta(\cI)$ and $\bq_t \in \Delta(\cJ)$. At the end of a round,
the first player does not observe $J_t$ nor $r(I_t,J_t)$ but only a signal.
There is a finite set $\cH$ of possible signals; the feedback $S_t$ that is given to the
first player is drawn at random according to the distribution
$H(I_t,J_t)$, where the mapping $H : \cI \times \cJ \to \Delta(\cH)$ is known by the first player.

\begin{example}
Examples of such partial monitoring games are provided by, e.g., \citemor{CeLuSt06}, among which we can cite the apple tasting problem,
the label-efficient prediction constraint, and the multi-armed bandit settings.
\end{example}

Some additional notation will be useful.
We denote by $R$ the norm of (the linear extension of) $r$,
\[
R = \max_{(i,j) \in \cI \times \cJ} \bnorm[ r(i,j) ]_2\,.
\]
The cardinalities of the finite sets $\cI$, $\cJ$, and $\cH$
will be referred to as $N_{\cI}$, $N_{\cJ}$, and $N_{\cH}$.

Definition~\ref{def:pure} can be extended as follows in this setting; the only new ingredient
is the signaling structure, the aim is unchanged.

\begin{definition}
Let $\cC \subseteq \R^d$ be some set; $\cC$ is $r$--approachable for the signaling structure $H$ if there exists a strategy
of the first player such that, for all $\eps > 0$,
there exists an integer $T_\eps$ such that for all strategies of the second player,
\[
\P \left\{ \forall \, T \geq T_\eps, \quad
\inf_{c \in \cC} \ \norm[c - \frac{1}{T} \sum_{t=1}^T r(I_t,J_t)]_2 \,\, \leq \e \right\}
\ \geq 1-\eps\,.
\]
That is, the first player has a strategy that ensures that the sequence of his average vector-valued payoffs
converges to the set $\cC$ (uniformly with respect to the strategies of the second player),
even if he only observes the random signals $S_t$ as a feedback.
\end{definition}

\paragraph{Our contributions.}
A necessary and sufficient condition for $r$--approachability with the signaling structure $H$ was
stated and proved by \citemor{Per11}; we therefore need to indicate where our contribution lies.
First, both proofs are constructive but our strategy can be efficient (as soon as some projection
operator can be computed efficiently, e.g., in the cases of external and internal regret minimization
described below) whereas the one of~\citemor{Per11} relies on auxiliary
strategies that are calibrated and that require a grid that is progressively refined
(leading to a step complexity that is exponential in the number $T$ of past steps); the latter construction
is in essence the one used in Section~\ref{sec:robappr-gal}.
Second, we are able to exhibit convergence rates. Third, as far as elegancy is concerned,
our proof is short, compact, and more direct than the one of~\citemor{Per11},
which relied on several layers of notations (internal regret in games with partial
monitoring, calibration of auxiliary strategies, etc.). \medskip

\subsection{Statement of the necessary and sufficient condition for approachability in games with partial monitoring.}
To recall the mentioned approachability condition of~\citemor{Per11}
we need some additional notation: for all $\bq \in \Delta(\cJ)$,
we denote by $\tH(\bq)$ the element in $\Delta(\cH)^\cI$ defined as follows. For all $i \in \cI$,
its $i$--th component is given by the convex combination of probability distributions
over $\cH$
\[
\tH(\bq)_i=H(i,\bq) = \sum_{j \in \cJ} q_j H(i,j)\,.
\]
Finally, we denote by $\cF$ the convex set of feasible vectors of probability distributions over $\cH$:
\[
\cF = \Bigl\{ \tH(\bq) : \ \ \bq \in \Delta(\cJ) \Bigr\}\,.
\]
A generic element of $\cF$ will be denoted by $\sigma \in \cF$ and
we define the set-valued function $\om$, for all $\bp \in \Delta(\cI)$ and $\sigma \in \cF$, by
\[
\om(\bp,\sigma) = \bigl\{ r(\bp,\bq') : \ \ \bq' \in \Delta(\cJ) \
\mbox{\rm such that} \ \tH(\bq') = \sigma \bigr\}\,.
\]

The necessary and sufficient condition exhibited by~\citemor{Per11} for the $r$--approachability of $\cC$ with the signaling
structure $H$ can now be recalled. In the sequel we will refer to this condition as Condition~{\apm},
an acronym that stands for ``approachability with partial monitoring.''

\begin{condition}[referred to as Condition (APM)]
\label{cd:1}
The signaling structure $H$, the vector-payoff function $r$, and the set $\cC$ satisfy
\[
\forall \, \bq \in \Delta(\cJ), \ \ \exists \, \bp \in \Delta(\cI), \ \
\forall \, \bq' \in \Delta(\cJ), \qquad  \tH(\bq) = \tH(\bq') \ \ \Rightarrow \ \
r(\bp,\bq') \in \cC\,.
\]
The condition can be equivalently reformulated as
\begin{equation}
\label{eq:APM}
\tag{\APM}
\forall \, \sigma \in \cF, \ \ \exists \, \bp \in \Delta(\cI), \quad \qquad
\om(\bp,\sigma) \subseteq \cC\,.
\end{equation}
\end{condition}

\paragraph{This condition is necessary.}
The subsequent sections show (in a constructive way)
that Condition~{\apm} is sufficient for $r$--approachability
of closed convex sets $\cC$ given the signaling structure $H$.
That this condition is necessary was already proved in Section~3.1 of~\citemor{Per11}.

\subsection{Links with robust approachability.}
\label{sec:links-calibr}

As will become clear in the proof of Theorem~\ref{TheoGeneralCase},
the key in our problem will be to ensure the robust approachability of $\cC$
with the following non-linear set-valued payoff function,
that is however concave--convex in the sense of Definition~\ref{def:convconc}.

\begin{lemma}
\label{lm:cvxconc}
The function
\[
(\bp,\bq) \in \Delta(\cI) \times \Delta(\cJ) \,\, \longmapsto \,\,
\om\bigl( \bp, \, H(\bq) \bigr)\,.
\]
is concave in its first argument and
convex in its second argument.
\end{lemma}

Unfortunately, efficient strategies for robust approachability were only proposed in the linear case,
not in the concave--convex case. But we illustrate in the next example (and provide a general theory in
the next section) how working in lifted spaces can lead to linearity and hence to efficiency.

\begin{example}
\label{ex:1}
We consider a game in which the second player (the column player) can force the first player (the row player)
to play a game of matching pennies in the dark by choosing actions $L$ or $M$; in the matrix below, the real numbers denote the payoff
while $\sigdark$ and $\sigother$ denote the two possible signals. The respective sets of actions are
$\cI = \{T, \, B\}$ and $\cJ = \{L, \, M, \, R\}$.
\begin{center}
\begin{tabular}{ccccc}
\hline
& & $L$ & $M$ & $R$ \\
\hline
$T$ & & $1$ / $\sigdark$ & $-1$ / $\sigdark$ & $2$ / $\sigother$ \\
$B$ & & $-1$ / $\sigdark$ & $1$ / $\sigdark$ & $3$ / $\sigother$ \\
\hline
\end{tabular}
\end{center}
\end{example}

In this example we only study the mapping $\bp \mapsto \om(\bp,\sigdark)$ and show that it is
piecewise linear on $\Delta(\cI)$, thus, is induced by a linear mapping defined on a lifted space.

We introduce a set $\cA = \{ \bp_T, \, \bp_B, \, \bp_{1/2} \}$ of possibly mixed actions extending the set $\cI = \{ T, \, B\}$ of pure actions;
the set $\cA$ is composed of
\[
\bp_T = \delta_T, \quad \bp_B = \delta_B, \quad \mbox{and} \quad
\bp_{1/2} = \frac{1}{2}\delta_T+\frac{1}{2}\delta_B\,.
\]
Each mixed action in $\Delta(\cI)$ can be uniquely written as $\bp_\lambda = \lambda \, \delta_B + (1-\lambda) \, \delta_T$
for some $\lambda \in [0,1]$.
Now, for $\lambda \geq 1/2$, first,
\[
\bp_\lambda = (2\lambda-1) \, \delta_B + \bigl( 1 - (2\lambda-1) \bigr) \, \bp_{1/2}\,;
\]
second, by definition of $\om$,
\[
\om \bigl( \bp_\lambda, \, \sigdark \bigr) = [1-2\lambda, \, 2\lambda-1]\,;
\]
since in particular $\om \bigl( \bp_{1/2}, \, \sigdark \bigr) = \{ 0 \}$
and $\om(\delta_B,\sigdark) = [-1,1]$,
we have the convex decomposition
\[
\om \bigl( \bp_\lambda, \, \sigdark \bigr) = (2\lambda-1) \, \om(\delta_B,\sigdark)
+ \bigl( 1 - (2\lambda-1) \bigr) \, \om(\bp_{1/2},\sigdark)\,,
\]
which can be restated as
\[
\om \bigl( \bp_\lambda, \, \sigdark \bigr) =
\om \Bigl( (2\lambda-1) \, \delta_B + \bigl( 1 - (2\lambda-1) \bigr) \, \bp_{1/2},
\, \sigdark \Bigr) = (2\lambda-1) \, \om(\delta_B,\sigdark)
+ \bigl( 1 - (2\lambda-1) \bigr) \, \om(\bp_{1/2},\sigdark)\,.
\]
That is, $\om( \,\cdot\,, \, \sigdark)$ is linear on the subset of $\Delta(\cI)$ corresponding
to mixed actions $\bp_\lambda$ with $\lambda \geq 1/2$.

A similar property holds the subset of distributions with $\lambda \leq 1/2$, so that
we have proved that $\om( \,\cdot\,, \, \sigdark)$ is piecewise linear on $\Delta(\cI)$.

The linearity on a lifted space comes from the following observation:
$\om$ is induced by the linear extension to $\Delta(\cA)$ of the
restriction of $\om$ to $\cA$ (see Definition~\ref{def:oom} for a more formal statement).

\section{Application of robust approachability to games with partial monitoring: for a particular class of games
encompassing regret minimization.}
\label{sec:main}

In this section we consider the case where the signaling structure has some special
properties described below (linked to linearity properties on lifted spaces)
and that can be exploited to get efficient strategies.
The case of general signaling structures is then considered in Section~\ref{sec:galstruc}
but the particular class of games considered here is already rich enough to encompass
the minimization of external and internal regret.

\subsection{Approachability in bi-piecewise linear games.}
\label{se:bipiecewise}

To define bi-piecewise linearity of a game, we start from a technical lemma that shows that $\om(\bp,\sigma)$ can be written as a {\em finite} convex combination of sets of the form $\om(\bp,b)$, where $b$ belongs to some finite set $\cB \subseteq \cF$ that depends on the game.
Under the additional assumption of piecewise linearity of
the thus-defined mappings $\om(\,\cdot\,,b)$, we
then describe a (possibly) efficient strategy for approachability followed by convergence rate guarantees.

\subsubsection{Bi-piecewise linearity of a game -- A preliminary technical result.}

\begin{lemma}
\label{lm:geocorr}
For any game with partial monitoring,
there exists a finite set $\cB \subset \cF$ and a piecewise-linear (injective) mapping $\Phi: \cF \to \Delta(\cB)$ such that
\[
\forall \, \sigma \in \cF, \quad \forall \, \bp \in \Delta(\cI), \qquad \quad
\om(\bp,\sigma) = \sum_{b \in \cB} \Phi_b(\sigma) \, \om(\bp,b)\,,
\]
where we denoted the convex weight vector $\Phi(\sigma) \in \Delta(\cB)$ by $\bigl( \Phi_b(\sigma) \bigr)_{b \in \cB}$.
\end{lemma}

\begin{proof}
Since $\tH$ is linear on the polytope $\Delta(\cJ)$,
Proposition~2.4 in~\citemor{RZ96} implies
that its inverse application $\tH^{-1}$ is a piecewise linear mapping of $\cF$ into the subsets of $\Delta(\cJ)$. This
means that there exists a finite decomposition of $\cF$ into polytopes $\{P_1, \dots, P_K\}$ each on which $\tH^{-1}$ is linear.
Up to a triangulation (see, e.g., Chapter~14 in~\cite{GR04}), we can assume that each $P_k$ is a simplex.
Denote by $\cB_k \subseteq \cF$ the set of vertices of $P_k$; then, the finite subset stated in the lemma is
\[
\cB= \bigcup_{k=1}^K \cB_k\,,
\]
the set of all vertices of all the simplices.

Fix any $\sigma \in \cF$. It belongs to some simplex $P_k$, so that there exists
a convex decomposition $\sigma = \sum_{b \in \cB_k} \lambda_b \, b$; this decomposition is unique within the simplex $P_k$.
If $\sigma$ belongs to two different simplices, then it actually belongs to their common face and
the two possible decompositions coincide (some coefficients $\lambda_b$ in the above decomposition are null). All in all,
with each $\sigma \in \cF$, we can associate a unique decomposition in $\cB$,
\[
\sigma = \sum_{b \in \cB} \Phi_b(\sigma) \, b\,,
\]
where the coefficients $\bigl( \Phi_b(\sigma) \bigr)_{b \in \cB}$ form a convex weight vector over $\cB$, i.e., belong
to $\Delta(\cB)$; in addition, $\Phi_b(\sigma) > 0$ only if $b \in \cB_k$, where $k$ is such that $\sigma \in P_k$.

Since $\tH^{-1}$ is linear on each simplex $P_1,\,\ldots,\,P_K$, we therefore get
\[
\tH^{-1}(\sigma) = \sum_{b \in \cB} \Phi_b(\sigma) \, \tH^{-1}(b)\,.
\]
Finally, the result is a consequence of the fact that
\[
\om(\bp,\sigma) = r\!\left(\bp, \, \tH^{-1}(\sigma)\right) = r\!\left(\bp,\sum_{b \in \cB} \Phi_b(\sigma) \, \tH^{-1}(b)\right)\,,
\]
which implies, by linearity of $r$, that
\[
\om(\bp,\sigma)=\sum_{b \in \cB} \Phi_b(\sigma) \, r\!\left(\bp,\tH^{-1}(b)\right) = \sum_{b \in \cB} \Phi_b(\sigma) \, \om(\bp,b)\,,
\]
which concludes the proof.
\end{proof}

\begin{remark}
\label{rk:Lip}
The proof shows that $\Phi$ is piecewise linear on a finite decomposition of $\cF$;
it is therefore Lipschitz on $\cF$.
We denote by $\kappa_\Phi$ its Lipschitz constant with respect to the $\ell^2$--norms.
\end{remark}

The main contribution  of this subsection (Definition~\ref{def:oom}) relies on the following additional assumption.

\begin{assumption}
\label{As:BPL}
A game is bi-piecewise linear if $\om(\,\cdot\,,b)$ is piecewise linear on $\Delta(\cI)$ for every $b \in \cB$.
\end{assumption}

Assumption~\ref{As:BPL} means that for all $b \in \cB$ there exists
a decomposition of $\Delta(\cI)$ into polytopes each on which $\om(\,\cdot\,,b)$ is linear.
Since $\cB$ is finite, there exists a finite number of such decompositions, and thus there exists a
decomposition to polytopes that refines all of them. (The latter is generated by the intersection of
all considered polytopes as $b$ varies.) By construction, every $\om(\,\cdot\,,b)$ is linear on any of the polytopes of this common decomposition.
We denote by $\cA \subset \Delta(\cI)$ the finite subset of all their vertices: a construction similar to the one used in the proof
of Lemma~\ref{lm:geocorr} (provided below) then leads to a piecewise linear (injective) mapping $\Theta: \Delta(\cI) \to \Delta(\cA)$, where $\Theta(\bp)$ is the decomposition of $\bp$ on the vertices of the polytope(s) of the decomposition to which it belongs, satisfying
\[
\forall \, b \in \cB, \quad \forall \, \bp \in \Delta(\cI), \qquad
\om(\bp,b) = \sum_{a \in \cA} \Theta_a(\bp) \, \om(a,b)\,,
\]
where we denoted the convex weight vector $\Theta(\bp) \in \Delta(\cA)$ by $\bigl( \Theta_a(\bp) \bigr)_{a \in \cA}$.
This, Lemma~\ref{lm:geocorr}, and Assumption~\ref{As:BPL}
show that on a lifted space, $\om$ coincides with a bi-linear mapping $\oom$, as is made formal
in the next definition.

\begin{definition}
\label{def:oom}
We denote by $\oom$ the linear extension to $\Delta(\cA \times \cB)$ of the restriction of $\om$ to $\cA \times \cB$,
so that for all $\bp \in \Delta(\cI)$ and $\sigma \in \cF$,
\[
\om(\bp,\sigma) = \oom \bigl( \Theta(\bp), \, \Phi(\sigma) \bigr)\,.
\]
\end{definition}

\subsubsection{Construction of a strategy to approach $\cC$.}

The approaching strategy for the original problem is based on a strategy $\Psi$ for $\oom$--approachability of $\cC$,
provided by Theorem~\ref{th:rob-appr}; we therefore first need to prove the existence
of such a $\Psi$.

\begin{lemma}
\label{lm:Psi}
Under Condition~{\apm},
the closed convex set $\cC$ is $\oom$--robust approachable.
\end{lemma}

\begin{proof}
We show that Condition~\eqref{eq:RAC} in Theorem~\ref{th:rob-appr} is satisfied, that is,
that for all $\by \in \Delta(\cB)$, there exists some $\bx \in \Delta(\cA)$ such that $\oom(\bx,\by) \subseteq \cC$.
With such a given $\by \in \Delta(\cB)$, we associate\footnote{Note however that we do not necessarily have that $\Phi(\sigma)$ and $\by$
are equal, as $\Phi$ is not a one-to-one mapping (it is injective but not surjective).} the feasible vector of signals
$\sigma = \sum_{b \in \cB} y_b \, b \in \cF$ and let $\bp$ be given by Condition~{\apm}, so that $\om(\bp,\sigma)\subseteq\cC$.
By linearity of $\oom$ (for the first equality),
by convexity of $\om$ in its second argument (for the first inclusion),
by Lemma~\ref{lm:geocorr} (for the second and fourth equalities),
by construction of $\cA$ (for the third equality),
\begin{eqnarray*}
\oom\bigl(\Theta(\bp),\by\bigr) = \sum_{a \in \cA} \Theta_a(\bp) \sum_{b \in \cB} y_b \, \om(a,b)
& \subseteq & \sum_{a \in \cA} \Theta_a(\bp)  \, \om(a,\sigma) = \sum_{a \in \cA} \Theta_a(\bp) \sum_{b \in \cB} \Phi_b(\sigma) \, \om(a,b)
\\
& = & \sum_{b \in \cB} \Phi_b(\sigma) \, \om(\bp,b) =
\om(\bp,\sigma) \subseteq \cC\,,
\end{eqnarray*}
which concludes the proof.
\end{proof}

We consider the strategy described in Figure~\ref{fig:strat}.
\begin{figure}[t]
\rule{\linewidth}{.5pt}{}\\
Approaching Strategy in Games with Partial Monitoring\\
\rule{\linewidth}{.5pt}
{\small
\emph{Parameters}: an integer block length $L \geq 1$, an exploration parameter $\gamma \in [0,1]$,
a strategy $\Psi$ for $\oom$--robust approachability
of~$\cC$ \smallskip \\
\emph{Notation}: $\bu \in \Delta(\cI)$ is the uniform distribution over $\cI$,
$P_\cF$ denotes the projection operator in $\ell^2$--norm of $\R^{\cH \times \cI}$ onto $\cF$ \smallskip \\
\emph{Initialization}: compute the finite set $\cB$ and the mapping $\Phi : \cF \to \Delta(\cB)$
of Lemma~\ref{lm:geocorr}, compute the finite set $\cA$ and the mapping $\Theta : \Delta(\cI) \to \Delta(\cA)$
defined based on Assumption~\ref{As:BPL},
pick an arbitrary $\bth_1 \in \Delta(\cA)$ \medskip \\
\emph{For all blocks} $n = 1,2,\ldots$,
\begin{enumerate}
\item define $\bx_n = \sum_{a \in \cA} \theta_{n,a} \, a$ \ and \ $\bp_n = (1-\gamma) \, \bx_n + \gamma \, \bu$;
\item for rounds $t = (n-1)L+1,\,\ldots,\,nL$,
   \begin{enumerate}
   \item[2.1] draw an action $I_t \in \cI$ at random according to $\bp_n$;
   \item[2.2] get the signal $S_t$;
   \end{enumerate}
\item form the estimated vector of probability distributions over signals,
\[
\displaystyle{
\widetilde{\sigma}_n = \left( \frac{1}{L} \sum_{t=(n-1)L+1}^{nL} \,
\frac{\ind_{ \{ S_t = s \} } \ind_{ \{ I_t = i \} }}{p_{I_t,n}} \right)_{(i,s) \in \cI \times \cH};}
\]
\item compute the projection $\wh{\sigma}_n = P_\cF \bigl( \widetilde{\sigma}_n \bigr)$;
\item choose $\bth_{n+1} = \Psi \Bigl( \bth_1, \, \Phi \bigl( \wh{\sigma}_1 \bigr), \, \ldots, \,
\bth_n, \, \Phi \bigl( \wh{\sigma}_n \bigr) \Bigr).$
\end{enumerate}
}
\rule{\linewidth}{.5pt}
\caption{\label{fig:strat} The proposed strategy, which plays in blocks.}
\end{figure}
It forces exploration at a $\gamma$ rate, as is usual in situations with partial monitoring.
One of its key ingredient, that conditionally unbiased estimators are available, is extracted from Section~6 in the article
by \citemor{LMS08}:
in block $n$ we consider sums of elements of the form
\[
\wh{H}_t = \left( \frac{\ind_{ \{ S_t = s \} } \ind_{ \{ I_t = i \} }}{p_{I_t,n}} \right)_{(i,s) \in \cI \times \cH}
\in \R^{\cH \times \cI};
\]
averaging over the respective random draws of $I_t$ and $S_t$ according to $\bp_n$ and $H(I_t,J_t)$,
i.e., taking the conditional expectation $\E_t$ with respect to $\bp_n$ and $J_t$, we get
\begin{equation}
\label{eq:extracted-source}
\E_t \bigl[ \wh{H}_t \bigr] = \tH \bigl( \delta_{J_t} \bigr).
\end{equation}
Indeed, the conditional expectation of the component $i$ of $\wh{H}_t$
equals
\[
\E_t \! \left[ \left( \frac{\ind_{ \{ S_t = s \} } \ind_{ \{ I_t = i \} }}{p_{I_t,n}} \right)_{\! s \in \cH} \right]
= \E_t \! \left[ \frac{H(I_t,J_t) \, \ind_{ \{ I_t = i \} }}{p_{I_t,n}} \right]
= \frac{H(i,J_t)}{p_{i,n}} \,\, \E_t \bigl[ \ind_{ \{ I_t = i \} } \bigr]
= H(i,J_t)\,,
\]
where we first took the expectation over the random draw of $S_t$ (conditionally to $\bp_n$, $J_t$, and $I_t$)
and then over the one of $I_t$.
Consequently, concentration-of-the-measure arguments can show that for $L$ large enough,
\[
\widetilde{\sigma}_n = \frac{1}{L} \sum_{t=(n-1)L+1}^{nL} \wh{H}_t
\qquad \mbox{is close to} \qquad \tH\bigl( \wh{\bq}_n \bigr)\,,
\qquad \mbox{where} \quad
\wh{\bq}_n = \frac{1}{L} \sum_{t=(n-1)L+1}^{nL} \delta_{J_t}\,.
\]
Actually, since $\cF \subseteq \Delta(\cH)^{\cI}$, we have a natural embedding of
$\cF$ into $\R^{\cH \times \cI}$ and we can define $P_\cF$, the convex projection operator onto~$\cF$ (in $\ell^2$--norm).
Instead of using directly $\widetilde{\sigma}_n$, we consider in our strategy
$\wh{\sigma}_n = P_\cF \bigl( \widetilde{\sigma}_n \bigr)$, which is even closer to $\tH\bigl( \wh{\bq}_n \bigr)$.

More precisely, the following result can be extracted from the proof of Theorem~6.1 in~\citemor{LMS08}. The proof
is provided in Appendix \ref{sec:appLemmaProof}.

\begin{lemma}
\label{lm:extracted}
With probability $1-\delta$,
\[
\norm[ \wh{\sigma}_n - \tH\bigl( \wh{\bq}_n \bigr) ]_2 \leq \sqrt{N_{\cI} N_{\cH}} \left( \sqrt{\frac{2 N_{\cI}}{\gamma L}
\ln \frac{2 N_\cI N_\cH}{\delta}} + \frac{1}{3} \frac{N_\cI}{\gamma L} \ln \frac{2 N_\cI N_\cH}{\delta} \right).
\]
\end{lemma}

\subsubsection{A performance guarantee for the strategy of Figure~\ref{fig:strat}.}
For the sake of simplicity, we provide
first a performance bound for fixed parameters $\gamma$ and $L$ tuned as functions of $T$.
Adaptation to $T \to \infty$ is then described in the next section; note that
it cannot be performed by simply proceeding in regimes, as the approachability guarantees
offered by the second part of the theorem are only at time round $T$. (This is so because the considered
strategy depends on $T$ via the parameters $\gamma$ and $L$.)

\begin{theorem}
\label{TheoGeneralCase}
Consider a closed convex set $\cC$ and
a game $(r,H)$ for which Condition~{\apm}
is satisfied and that is bi-piecewise linear in the
sense of Assumption~\ref{As:BPL}.
Then, for all $T \geq 1$, the strategy of Figure~\ref{fig:strat}, run with
parameters $\gamma \in [0,1]$ and $L \geq 1$
and fed with a strategy $\Psi$ for $\oom$--approachability of $\cC$
(provided by Lemma~\ref{lm:Psi}) is such that, with probability at least
$1 - \delta$,
\begin{multline}
\nonumber
\inf_{c \in \cC} \ \norm[c - \frac{1}{T} \sum_{t=1}^T r(I_t,J_t)]_2 \,\, \leq \ \frac{2L}{T} R +
4R \sqrt{\frac{\ln \bigl( (2T)/ (L\delta) \bigr)}{T}} + 2 \gamma R + \frac{2R}{\sqrt{T/L-1}} \sqrt{N_{\cA} N_{\cB}} \\
+ R \kappa_\Phi \sqrt{N_{\cI} N_{\cH} N_{\cA}} \left( \sqrt{\frac{2 N_{\cI}}{\gamma L}
\ln \frac{2 N_\cI N_\cH T}{L\delta}} + \frac{1}{3} \frac{N_\cI}{\gamma L} \ln \frac{2 N_\cI N_\cH T}{L\delta} \right).
\end{multline}
In particular, for all $T \geq 1$,
the choices of $L = \bigl\lceil T^{3/5} \bigr\rceil$ and $\gamma = T^{-1/5}$ imply that with probability at least $1-\delta$,
\[
\inf_{c \in \cC} \ \norm[c - \frac{1}{T} \sum_{t=1}^T r(I_t,J_t)]_2 \,\, \leq \,\,
\square \left(T^{-1/5}\sqrt{\ln \frac{T}{\delta}} + T^{-2/5} \ln \frac{T}{\delta} \right)
\]
for some constant $\square$ depending only on $\cC$ and on the game $(r,\,H)$ at hand.
\end{theorem}

The efficiency of the strategy of Figure~\ref{fig:strat} depends on whether it can be fed
with an efficient approachability strategy $\Psi$, which in turn depends on the respective geometries
of $\om$ and $\cC$, as was indicated before the statement of Theorem~\ref{th:rob-appr}.
(Note that the projection onto $\cF$ can be performed in polynomial time, as the latter closed convex set
is defined by finitely many linear constraints, and that the computation of $\cA$, $\cB$, and $\oom$ can be performed
beforehand.) In any case, the per-round complexity is constant (though possibly large).

\begin{proof}
{\allowdisplaybreaks
We write $T$ as $T = NL+k$ where $N$ is an integer and $0 \leq k \leq L-1$
and will show successively that (possibly with overwhelming probability only) the following statements hold.
\begin{eqnarray}
\label{eq:term1}
\frac{1}{T} \sum_{t=1}^T r(I_t,J_t) & \quad \mbox{is close to} \quad &
\frac{1}{NL} \sum_{t=1}^{NL} r(I_t,J_t)\,; \\
\label{eq:term2}
\frac{1}{NL} \sum_{t=1}^{NL} r(I_t,J_t) & \quad \mbox{is close to} \quad &
\frac{1}{N} \sum_{n=1}^{N} r \bigl( \bp_n, \, \wh{\bq}_n \bigr)\,; \\
\label{eq:term3}
\frac{1}{N} \sum_{n=1}^{N} r \bigl( \bp_n, \, \wh{\bq}_n \bigr) & \quad \mbox{is close to} \quad &
\frac{1}{N} \sum_{n=1}^{N} r \bigl( \bx_n, \, \wh{\bq}_n \bigr)\,; \\
\nonumber
\frac{1}{N} \sum_{n=1}^{N} r \bigl( \bx_n, \, \wh{\bq}_n \bigr)
= \frac{1}{N} \sum_{n=1}^{N} \sum_{a \in \cA} \theta_{n,a} \, r \bigl( a, \, \wh{\bq}_n \bigr)
& \quad \mbox{belongs to the set} \quad &
\frac{1}{N} \sum_{n=1}^{N} \sum_{a \in \cA} \theta_{n,a} \,
\om \Bigl( a, \, \tH\bigl( \wh{\bq}_n \bigr) \Bigr)\,; \\
\nonumber
\frac{1}{N} \sum_{n=1}^{N} \sum_{a \in \cA} \theta_{n,a} \,
\om \Bigl( a, \, \tH\bigl( \wh{\bq}_n \bigr) \Bigr)
& \quad \mbox{is equal to the set} \quad &
\frac{1}{N} \sum_{n=1}^{N} \oom \biggl( \bth_n, \, \Phi \Bigl( \tH\bigl( \wh{\bq}_n \bigr) \Bigr) \biggr)\,; \\
\label{eq:term4}
\frac{1}{N} \sum_{n=1}^{N} \oom \biggl( \bth_n, \, \Phi \Bigl( \tH\bigl( \wh{\bq}_n \bigr) \Bigr) \biggr) & \quad \mbox{is close to the set} \quad &
\frac{1}{N} \sum_{n=1}^{N} \oom \Bigl( \bth_n, \, \Phi \bigl( \wh{\sigma}_n \bigr) \Bigr)\,; \\
\label{eq:term5}
\frac{1}{N} \sum_{n=1}^{N} \oom \Bigl( \bth_n, \, \Phi \bigl( \wh{\sigma}_n \bigr) \Bigr) & \quad \mbox{is close to the set} \quad &
\cC\,;
\end{eqnarray}
}
where we recall that the notation $\wh{\bq}_n$ was defined above
and is referring to the empirical distribution of the $J_t$ in the $n$--th block.
Actually, we will show below the numbered statements only. The first unnumbered statement
is immediate by the definition of $\bx_n$, the linearity of $r$, and the very definition of $\om$;
while the second one follows from Definition~\ref{def:oom}:
\[
\frac{1}{N} \sum_{n=1}^{N} \sum_{a \in \cA} \theta_{n,a} \, \om \Bigl( a, \, \tH\bigl( \wh{\bq}_n \bigr) \Bigr)
= \frac{1}{N} \sum_{n=1}^{N} \sum_{(a,b) \in \cA \times \cB} \theta_{n,a} \, \Phi_b\Bigl( \tH\bigl( \wh{\bq}_n \bigr) \Bigr) \,
\om(a,b) \\ = \frac{1}{N} \sum_{n=1}^{N} \oom \biggl( \bth_n, \, \Phi\Bigl( \tH\bigl( \wh{\bq}_n \bigr) \Bigr) \biggr)\,.
\]

\steps{1}{eq:term1} A direct calculation decomposing the sum over $T$ elements
into a sum over the $NL$ first elements and the $k$ remaining ones shows that
\[
\norm[ \frac{1}{T} \sum_{t=1}^T r(I_t,J_t) - \frac{1}{NL} \sum_{t=1}^{NL} r(I_t,J_t) ]_2
\leq R \left( \frac{k}{T} + \left( \frac{1}{NL} - \frac{1}{T} \right) NL \right) = \frac{2k}{T} R \leq \frac{2L}{T} R\,.
\]

\steps{2}{eq:term2} We note that by defining $\E_t$ the conditional expectation
with respect to $(I_1,S_1,J_1)$, $\ldots$, $(I_{t-1},S_{t-1},J_{t-1})$ and $J_t$, which
fixes the values of the distribution $\bp'_t$ of $I_t$ and the value of $J_t$, we have
\[
\E_t \bigl[ r(I_t,J_t) \bigr] = r(\bp'_t,J_t)\,.
\]
We note that by definition of the forecaster, $\bp'_t = \bp_n$ if $t$ belongs to the $n$--th block.
By a version of the Hoeffding-Azuma inequality for sums of Hilbert space-valued
martingale differences stated as\footnote{Together with the fact that
$\sqrt{u} \, e^{-u} \leq e^{- u/2}$ for all $u \geq 0$.} Lemma~3.2 in~\citemor{CW96}, we therefore get
that with probability at least $1-\delta$,
\[
\norm[ \frac{1}{NL} \sum_{t=1}^{NL} r(I_t,J_t) -
\frac{1}{N} \sum_{n=1}^{N} r \bigl( \bp_n, \, \wh{\bq}_n \bigr) ]_2 \leq 4R \sqrt{\frac{\ln (2/\delta)}{T}}\,.
\]

\steps{3}{eq:term3} Since by definition $\bp_n = (1-\gamma) \, \bx_n + \gamma \, \bu$, we
get
\[
\norm[ \frac{1}{N} \sum_{n=1}^{N} r \bigl( \bp_n, \, \wh{\bq}_n \bigr) -
\frac{1}{N} \sum_{n=1}^{N} r \bigl( \bx_n, \, \wh{\bq}_n \bigr)]_2 \leq 2 \gamma R\,.
\]

\steps{4}{eq:term4} We fix a given block $n$. Lemma~\ref{lm:extracted} indicates
that with probability $1-\delta$,
\begin{equation}
\label{eq:extracted}
\norm[ \wh{\sigma}_n - \tH\bigl( \wh{\bq}_n \bigr) ]_2 \leq \sqrt{N_{\cI} N_{\cH}} \left( \sqrt{\frac{2 N_{\cI}}{\gamma L}
\ln \frac{2 N_\cI N_\cH}{\delta}} + \frac{1}{3} \frac{N_\cI}{\gamma L} \ln \frac{2 N_\cI N_\cH}{\delta} \right).
\end{equation}
Since $\Phi$ is Lipschitz (see Remark~\ref{rk:Lip}), with a Lipschitz constant in $\ell^2$--norms
denoted by $\kappa_\Phi$, we get that
with probability $1-\delta$,
\[
\norm[ \Phi \bigl( \wh{\sigma}_n \bigr) -
\Phi \Bigl( \tH\bigl( \wh{\bq}_n \bigr) \Bigr) ]_2 \leq \kappa_\Phi \sqrt{N_{\cI} N_{\cH}} \left( \sqrt{\frac{2 N_{\cI}}{\gamma L}
\ln \frac{2 N_\cI N_\cH}{\delta}} + \frac{1}{3} \frac{N_\cI}{\gamma L} \ln \frac{2 N_\cI N_\cH}{\delta} \right).
\]
By a union bound, the above bound holds for all blocks $n = 1,\ldots,N$ with probability at least $1 - N\delta$.
Finally, an application of Lemma~\ref{lm:DVT} shows that
\[
\frac{1}{N} \sum_{n=1}^{N} \oom \biggl( \bth_n, \, \Phi \Bigl( \tH\bigl( \wh{\bq}_n \bigr) \Bigr) \biggr)
\qquad \mbox{is in a $\epsilon_T$--neighborhood (in $\ell^2$--norm) of} \qquad
\frac{1}{N} \sum_{n=1}^{N} \oom \Bigl( \bth_n, \, \Phi \bigl( \wh{\sigma}_n \bigr) \Bigr)\,,
\]
where
\[
\epsilon_T = R\sqrt{N_\cB} \,\, \times \,\, \kappa_\Phi \sqrt{N_{\cI} N_{\cH}} \left( \sqrt{\frac{2 N_{\cI}}{\gamma L}
\ln \frac{2 N_\cI N_\cH}{\delta}} + \frac{1}{3} \frac{N_\cI}{\gamma L} \ln \frac{2 N_\cI N_\cH}{\delta} \right).
\]

\steps{5}{eq:term5} Since $\cC$ is $\oom$--robust approachable and by definition of the choices of the $\bth_n$
in Figure~\ref{fig:strat}, we get by (the proof of the sufficiency part of) Theorem~\ref{th:rob-appr}, with probability 1,
\[
\inf_{c \in \cC} \ \norm[c - \frac{1}{N} \sum_{n=1}^{N} \oom \Bigl( \bth_n, \, \Phi \bigl( \wh{\sigma}_n \bigr) \Bigr)]_2 \leq
\frac{2R}{\sqrt{N}} \sqrt{N_{\cA} N_{\cB}} \leq \frac{2R}{\sqrt{T/L-1}} \sqrt{N_{\cA} N_{\cB}}\,,
\]
since $T/L \leq N + k/L \leq N+1$.

\textbf{Conclusion of the proof.}
The proof is concluded by putting the pieces together, thanks to a triangle inequality and
by considering $L\delta/T \leq \delta/(N+1)$ instead of $\delta$.
\end{proof}

\subsubsection{Uniform guarantees over time for a time-adaptive version of the strategy of Figure~\ref{fig:strat}.}
We present here a variant of the strategy of Figure~\ref{fig:strat} for which
the lengths $L_n$ of blocks $n$ and the exploration rates $\gamma_n$ are no longer constant.
To do so, we need the following generalization of Theorem~\ref{th:Bla} to polynomial averages;
this result is of independent interest. We only state the result for mixed actions taken and observed, but the
generalization for pure actions follows easily.


Consider the setting of Theorem~\ref{th:Bla}.
The studied strategy relies on a parameter $\alpha \geq 0$. It plays an arbitrary $\bx_1$.
For $t \geq 1$, it forms at stage $t+1$ the vector-valued polynomial average
\[
\hat{m}^{\alpha}_t = \frac{1}{T_t^\alpha} \, \sum_{s=1}^t s^\alpha \, m(\bx_s, \by_s)
\qquad \quad \mbox{where} \qquad T_t^\alpha= \sum_{s=1}^t s^\alpha\,,
\]
computes its projection $c^{\alpha}_t$ onto $\cC$,
and resorts to a mixed action $\bx_{t+1}$ solving the minimax equation
\[
\min_{\bx \in \Delta(\cA)} \, \max_{\by \in \Delta(\cB)} \,\, \bigl\langle \hat{m}^{\alpha}_t - c^{\alpha}_t , \, m(\bx,\by) \bigr\rangle\,.
\]

\begin{theorem}\label{th:BlaAlpha}
We denote by $M$ a bound in norm over $m$, i.e.,
\[
\max_{(a,b) \in \cA \times \cB} \bnorm[m(a,b)]_2 \leq M\,.
\]
For all $\alpha\geq 0$, when $\cC$ is an approachable closed convex set,
the above strategy ensures that
for all strategies of the second player, with probability 1,
for all $T \geq 1$,
\begin{equation}
\label{eq:polyweightedappr}
\inf_{c \in \cC} \ \norm[c - \frac{1}{\sum_{t=1}^T t^\alpha} \sum_{t=1}^T t^\alpha \, m( \bx_t, \by_t)]_2 \ \leq 2M \frac{\sqrt{\sum_{t=1}^T t^{2\alpha}}}{\sum_{t=1}^T t^{\alpha}} \leq  \frac{2M K_{\alpha}}{\sqrt{T}}\,,
\end{equation}
where $K_{\alpha}$ is a constant depending only $\alpha$.
\end{theorem}

It is interesting to note that the convergence rate are independent of $\alpha$ and are the same as standard approachability ($1/\sqrt{T}$).

\begin{proof}
The proof is a slight modification of the one of Theorem~\ref{th:Bla}.
We denote by $d_t^\alpha$ the squared distance of $\hat{m}_t^\alpha$ to $\cC$,
\[
d_t^\alpha = \inf_{c\in\cC} \bnorm[c - \hat{m}_t^\alpha]^2 = \bnorm[c^{\alpha}_t - \hat{m}_t^\alpha]^2
\]
and use the shortcut notation $m_{t}=m(\bx_{t},\by_{t})$ for all $t \geq 1$. Then,
\begin{align*}
d_{t+1}^\alpha & \leq \bnorm[ \hat{m}_{t+1}^\alpha- c_t^\alpha  ]^2=\norm[ \hat{m}_t^\alpha - c_t^\alpha  + \frac{(t+1)^\alpha}{T_{t+1}^\alpha}\Big(m_{t+1}-\hat{m}_t^\alpha\Big)]^2\\
& \leq \bnorm[ \hat{m}_t^\alpha-c_t^\alpha ]^2+\frac{2(t+1)^\alpha}{T_{t+1}^\alpha}
\bigl\langle \hat{m}_t^\alpha - c_t^\alpha, \, m_{t+1}-m_t^\alpha \bigr\rangle
+ \left(\frac{(t+1)^\alpha}{T_{t+1}^\alpha}\right)^2 \bnorm[m_{t+1}-\hat{m}_t^\alpha]^2 \\
& \leq d_t^\alpha + \frac{2(t+1)^\alpha}{T_{t+1}^\alpha}
\underbrace{\bigl\langle \hat{m}_t^\alpha - c_t^\alpha, \, m_{t+1} - c_t^\alpha \bigr\rangle}_{\leq 0}
+ \frac{2(t+1)^\alpha}{T_{t+1}^\alpha}
\bigl\langle \hat{m}_t^\alpha - c_t^\alpha, \, c_t^\alpha-m_t^\alpha \bigr\rangle
+ \left(\frac{(t+1)^\alpha}{T_{t+1}^\alpha}\right)^2 4M^2 \\
& \leq d_t^\alpha\left(1-\frac{2(t+1)^\alpha}{T_{t+1}^\alpha}\right) + \left(\frac{(t+1)^\alpha}{T_{t+1}^\alpha}\right)^2 4M^2,
\end{align*}
where we used in the third inequality the same convex projection inequality as in the proof of Theorem~\ref{th:Bla}.

The first inequality in (\ref{eq:polyweightedappr}) then follows by induction: the bound $2M$ for $t=1$ is by boundedness of $m$.
If the stated bound holds for $d_t^\alpha$, then
\[
d_{t+1}^\alpha \leq \left( 2M \frac{\sqrt{\sum_{s=1}^t s^{2\alpha}}}{\sum_{s=1}^t s^{\alpha}} \right)^{\!\! 2}
\left(1-\frac{2(t+1)^\alpha}{T_{t+1}^\alpha}\right) +
\left(\frac{(t+1)^\alpha}{T_{t+1}^\alpha}\right)^2 4M^2
\leq 4M^2 \, \frac{\sum_{s=1}^{t+1} s^{2\alpha}}{\bigl( T_{t+1}^\alpha \bigr)^2}\,,
\]
as desired, since
\[
\frac{1}{\bigl( T_t^\alpha \bigr)^2} \left(1-\frac{2(t+1)^\alpha}{T_{t+1}^\alpha}\right)
= \frac{1}{T_{t+1}^\alpha \, \bigl( T_t^\alpha \bigr)^2} \bigl( T_t^\alpha - (t+1)^\alpha \bigr)
\leq \frac{1}{T_{t+1}^\alpha \, \bigl( T_t^\alpha \bigr)^2} \, \frac{\bigl( T_t^\alpha \bigr)^2 - (t+1)^{2\alpha}}{T_t^\alpha
+ (t+1)^\alpha} \leq \frac{1}{\bigl( T_{t+1}^\alpha \bigr)^2}\,.
\]

The second inequality  in (\ref{eq:polyweightedappr})  can be proved as follows. First, for all $\alpha \geq 0$, by comparing sums and integrals, we get
that for all $t \geq 1$,
\[
\frac{t^{\alpha+1}}{\alpha + 1} = \int_0^t s^\alpha \,\mbox{d}s
\leq \sum_{s=1}^t s^\alpha \leq \int_1^{t+1} s^\alpha \,\mbox{d}s \leq
\frac{(t+1)^{\alpha+1}}{\alpha + 1} \leq \frac{(2t)^{\alpha+1}}{\alpha + 1}\,.
\]
Therefore,
\[
\frac{\sqrt{\sum_{s=1}^t s^{2\alpha}}}{\sum_{s=1}^t s^{\alpha}}
\leq \frac{\alpha + 1}{\sqrt{2\alpha +1}} \, \frac{\sqrt{(2t)^{\alpha+1}}}{t^{\alpha+1}}
= K_\alpha \frac{1}{\sqrt{t}}
\]
for
\[
K_\alpha = \frac{\alpha + 1}{\sqrt{2\alpha +1}} \, \sqrt{2^{\alpha + 1}}\,.
\]
This concludes the proof.
\end{proof}

The extension to polynomially weighted averages can also be obtained in the context of robust approachability
as the key to Theorem~\ref{th:rob-appr} is Lemma~\ref{lm:apprtorobappr}, which indicates that to get
robust approachability, it suffices to approach, in the usual sense, $\widetilde{\cC}$; both can thus be performed with polynomially
weighted averages.

Consider now the variant of the strategy of Figure~\ref{fig:strat} for which
the length of the $n$-th block, denoted by $L_n$, is equal to $n^{\alpha}$,
the exploration rate on this block comes at a rate $\gamma_n = n^{-\alpha/3}$
and $\Psi$ is an $\oom$-robust approachability strategy of $\cC$ with respect to polynomially weighted
averages with parameter $\alpha = 3/2$. We call it a time-adaptive version of this strategy;
note that it does not depend anymore on any time horizon $T$, hence guarantees can be obtained
for all $T$.

\begin{theorem}
\label{TheoGeneralCaseVariant}
The time-adaptive version of the strategy described in
Figure~\ref{fig:strat} (with $L_n=n^{\alpha}$ and $\gamma_n=n^{-\alpha/3}$ for
$\alpha=3/2$) ensures that, for all $T \geq 1$, with probability at least $1-\delta$,
\[
\inf_{c \in \cC} \ \norm[c - \frac{1}{T} \sum_{t=1}^T r(I_t,J_t)]_2 \,\, \leq \,\,
\square \left(T^{-1/5}\sqrt{\ln \frac{T}{\delta}} + T^{-2/5} \ln \frac{T}{\delta} \right)
\]
for some constant $\square$ depending only on $\cC$ and  the game $(r,\,H)$ at hand.
\end{theorem}

\begin{proof}
The proof follows closely the one of Theorem~\ref{TheoGeneralCase}.
We choose $N$ so as to write $T = T_N^{\alpha} + k$ where $0\leq k \leq L_{N+1}-1$. We adapt step~1 as follows,
\[
\norm[ \frac{1}{T} \sum_{t=1}^T r(I_t,J_t) - \frac{1}{T_N^{\alpha}} \sum_{t=1}^{T_N^\alpha} r(I_t,J_t) ]_2
\leq R \left( \frac{k}{T} + \left( \frac{1}{T_N^\alpha} - \frac{1}{T} \right) T_N^\alpha \right) = \frac{2k}{T} R \leq \frac{2L_{N+1}}{T} R\,.
\]
Second, as in step~2, we resort again to the Hoeffding-Azuma inequality for sums of Hilbert space-valued
martingale differences; with probability at least $1-\delta$,
\[
\norm[ \frac{1}{T^\alpha_N} \sum_{t=1}^{T_N^\alpha} r(I_t,J_t) -
\frac{1}{T_N^\alpha} \sum_{n=1}^{N} n^\alpha \, r \bigl( \bp_n, \, \wh{\bq}_n \bigr) ]_2
\leq 4R \sqrt{\frac{\ln (2/\delta)}{T_N^\alpha}} \leq 4R \sqrt{\frac{\ln (2/\delta)}{T}} \,.
\]
In view of the choice $\gamma_n = n^{-\alpha/3}$, step~3 translates here to
\[
\norm[ \frac{1}{T_N^\alpha} \sum_{n=1}^{N} n^\alpha \, r\bigl( \bp_n, \, \wh{\bq}_n \bigr) -
\frac{1}{T_N^\alpha} \sum_{n=1}^{N} n^\alpha \, r\bigl( \bx_n, \, \wh{\bq}_n \bigr)]_2
\leq 2R \, \frac{\sum_{n=1}^{N} n^\alpha \gamma_n }{T_N^\alpha}
= 2R \, \frac{\sum_{n=1}^{N} n^{2\alpha/3}}{T_N^\alpha}
= 2 R \frac{T_N^{(2\alpha/3)}}{T_N^\alpha}\,.
\]
The same argument as the one at the beginning of the proof of Theorem~\ref{TheoGeneralCase}
shows that
\[
\frac{1}{T_N^\alpha} \sum_{n=1}^{N} n^\alpha \, r\bigl( \bx_n, \, \wh{\bq}_n \bigr)
\,\, \in \,\,
\frac{1}{T_N^\alpha} \sum_{n=1}^{N} n^\alpha \, \oom \biggl( \bth_n, \, \Phi \Bigl( \tH\bigl( \wh{\bq}_n \bigr) \Bigr) \biggr).
\]
Step~4 starts also by an application of Lemma~\ref{lm:extracted} together with the Lipschitzness of $\Phi$
to get that for all regimes $n = 1,\ldots,N$,
with probability at least $1-\delta$,
\[
\norm[ \Phi \bigl( \wh{\sigma}_n \bigr) -
\Phi \Bigl( \tH\bigl( \wh{\bq}_n \bigr) \Bigr) ]_2 \leq \kappa_\Phi \sqrt{N_{\cI} N_{\cH}} \left( \sqrt{\frac{2 N_{\cI}}{\gamma_n L_n}
\ln \frac{2 N_\cI N_\cH}{\delta}} + \frac{1}{3} \frac{N_\cI}{\gamma_n L_n} \ln \frac{2 N_\cI N_\cH}{\delta} \right).
\]
By a union bound, the above bound holds for all regimes $n = 1,\ldots,N$ with probability at least $1 - N\delta$.
Then, an application of Lemma~\ref{lm:DVT} shows that
\[
\frac{1}{T_N^\alpha} \sum_{n=1}^{N} n^\alpha \, \oom \biggl( \bth_n, \, \Phi \Bigl( \tH\bigl( \wh{\bq}_n \bigr) \Bigr) \biggr)
\qquad \mbox{is in a $\epsilon_N$--neighborhood of} \qquad
\frac{1}{T_N^\alpha} \sum_{n=1}^{N} n^\alpha \, \oom \Bigl( \bth_n, \, \Phi \bigl( \wh{\sigma}_n \bigr) \Bigr)\,,
\]
where, substituting the values of $L_n = n^\alpha$ and $\gamma_n = n^{-\alpha/3}$,
\begin{align*}
\epsilon_N  & = R\sqrt{N_\cB} \,\, \times \,\, \kappa_\Phi \sqrt{N_{\cI} N_{\cH}} \,\, \frac{1}{T_N^\alpha} \sum_{n=1}^N n^\alpha \left( \sqrt{\frac{2 N_{\cI}}{\gamma_n L_n}
\ln \frac{2 N_\cI N_\cH}{\delta}} + \frac{1}{3} \frac{N_\cI}{\gamma_n L_n} \ln \frac{2 N_\cI N_\cH}{\delta} \right)\\
&  = R\sqrt{N_\cB} \,\, \times \,\, \kappa_\Phi \sqrt{N_{\cI} N_{\cH}} \left( \frac{T_N^{(2\alpha/3)}}{T_N^\alpha}\sqrt{2 N_{\cI}\ln \frac{2 N_\cI N_\cH}{\delta}} + \frac{T_N^{(\alpha/3)}}{T_N^\alpha} \, \frac{N_\cI}{3} \ln \frac{2 N_\cI N_\cH}{\delta} \right).
\end{align*}
It then suffices, as in step~5 of the original proof, to write the convergence rates for robust approachability guaranteed by
the strategy $\Psi$. By combining the result of Lemma~\ref{lm:apprtorobappr} with Theorem~\ref{th:BlaAlpha}
and Lemma~\ref{lm:DVT}, we get
\[
\inf_{c \in \cC} \ \norm[c - \frac{1}{T_n^\alpha}  \sum_{n=1}^{N} n^\alpha
\, \oom \Bigl( \bth_n, \, \Phi \bigl( \wh{\sigma}_n \bigr) \Bigr)]_2 \leq
\frac{2 R \, K_\alpha}{\sqrt{N}} \sqrt{N_{\cA} N_{\cB}}\,.
\]
Putting all things together and applying a union bound, we obtain that with probability at least $1-\delta$,
\begin{multline}
\nonumber
\qquad \inf_{c \in \cC} \ \norm[c - \frac{1}{T} \sum_{t=1}^T r(I_t,J_t)]_2 \\
= O \! \left( \frac{(N+1)^{\alpha}}{T}+\sqrt{\frac{\ln(N/\delta)}{T_n^\alpha}}+\frac{T_N^{(2\alpha/3)}}{T_N^\alpha}
+ \frac{T_N^{(2\alpha/3)}}{T_N^\alpha}
\sqrt{\ln \frac{N}{\delta}} + \frac{T_N^{(\alpha/3)}}{T_N^\alpha}\ln \frac{N}{\delta} + \frac{1}{\sqrt{N}} \right).
\end{multline}
Since (as proved at the end of Theorem~\ref{th:BlaAlpha}) $T_N^\beta \sim N^{\beta+1}/(\beta+1)$ for all $\beta \geq 0$,
we get that
\[
N \sim \bigl( (\alpha+1) T \bigr)^{1/(\alpha + 1)} \qquad \mbox{and}
\qquad T_N^\beta \sim \frac{N^{\beta+1}}{\beta+1} \sim \kappa_{\alpha,\beta} \, T^{(\beta+1)/(\alpha+1)}\,,
\]
where $\kappa_{\alpha,\beta}$ is a constant that only depends on $\alpha$ and $\beta$.
Choosing $\alpha = 3/2$ and substituting these equivalences ensures the result.
\end{proof}

\subsection{Application to regret minimization.}
\label{se:apptoregretmin}

In this section we analyze external and internal regret minimization in repeated games with partial monitoring from the approachability perspective.
We show how to---in particular---efficiently
minimize regret in both setups using the results developed for vector-valued games with partial monitoring;
to do so, we indicate why the assumption of bi-piecewise linearity (Assumption~\ref{As:BPL}) is satisfied.

\subsubsection{External regret.}
\label{sec:ER}

We consider in this section the framework and aim introduced by \citemor{Rus99} and studied, sometimes in special
cases, by \citemor{PiSc01}, \citemor{MaSh03}, \citemor{CeLuSt06}, \citemor{LMS08}. We show that our general strategy can be used for regret minimization.

Scalar payoffs are obtained (but not observed) by the first player, i.e., $d=1$:
the payoff function $r$ is a mapping $\cI \times \cJ \to \R$; we still denote by $R$ a bound on $|r|$.
We define in this section
\[
\wh{\bq}_T = \frac{1}{T} \sum_{t=1}^T \delta_{J_T}
\]
as the empirical distribution of the actions taken by the second player during the first $T$ rounds.
(This is in contrast with the notation $\wh{\bq}_T$ used in the previous section
to denote such an empirical distribution, but only taken within regime $n$.)

The external regret of the first
player at round $T$ equals by definition
\[
R^{\ext}_T = \max_{\bp \in \Delta(\cI)} \,\, \rho \Bigl( \bp, \tH \bigl( \wh{\bq}_T \bigr) \Bigr) -
\frac{1}{T} \sum_{t=1}^T r(I_t,J_t)\,,
\]
where $\rho : \Delta(\cI) \times \cF$ is defined as follows: for all $\bp \in \Delta(\cI)$ and $\sigma \in \cF$,
\[
\rho(\bp,\sigma) = \min\left\{ r(\bp,\bq)\, :  \, \bq \ \mbox{such that} \ \tH(\bq)= \sigma \right\}\,.
\]
The function $\rho$ is  continuous in its first argument and therefore the supremum
in the defining expression of $R^{\ext}_T$ is a maximum.

We recall briefly why, intuitively, this is the natural notion of external regret to consider in this case.
Indeed, the first term in the definition of $R^{\ext}_T$ is (close to) the
worst-case average payoff obtained by the first player when playing consistently a mixed
action $\bp$ against a sequence of mixed actions inducing on average the same laws on the signals
as the sequence of actions actually played.

The following result is an easy consequence of Theorem~\ref{TheoGeneralCaseVariant}, as is explained below;
it corresponds to the main result of \citemor{LMS08}, with the same convergence rate but
with a different strategy. (However, Section~2.3 of \citemor{Per11JMLR} exhibited an efficient strategy
achieving a convergence rate of order $T^{-1/3}$, which is optimal; a question that remains open is thus whether
the rates exhibited in Theorem~\ref{TheoGeneralCaseVariant} could be improved.)

\begin{corollary}
\label{cor:ER}
The first player has a
strategy such that for all $T$ and  all strategies of the second player, with probability
at least $1-\delta$,
\[
R^{\ext}_T \,\, \leq \,\,
\square \left(T^{-1/5}\sqrt{\ln \frac{T}{\delta}} + T^{-2/5} \ln \frac{T}{\delta} \right)
\]
for some constant $\square$ depending only on the game $(r,\,H)$ at hand.
\end{corollary}

The proof below is an extension to the setting of partial monitoring
of the original proof and strategy of~\citemor{Bla56b} for the case of external regret under full monitoring:
in the latter case the vector-payoff function $\underline{r}$ and the set $\cC$ considered in our proof are equal to the ones
considered by Blackwell. \medskip

\begin{proof}
We embed $\cF$ into $\R^{\cI \times \cH}$ so that in this proof we will be working in
the vector space $\R^d = \R \times \R^{\cI \times \cH}$.
We consider the closed convex set $\cC$ and the vector-valued payoff
function $\underline{r}$ respectively defined by
\[
\cC = \left\{ (z, \sigma) \in \R \times \cF : \ \ z \geq \max_{\bp \in \Delta(\cI)} \rho( \bp, \, \sigma) \right\}
\qquad \mbox{and} \qquad
\underline{r}(i,j) = \left[ \begin{array}{c} r(i,j) \\ \tH(\delta_j) \end{array} \right]\,,
\]
for all $(i,j) \in \cI \times \cJ$.

We first show that Condition~{\apm} is satisfied for the considered convex set $\cC$ and game
$(\underline{r},H)$. To do so, by continuity of $\rho$ in its first argument,
we associate with each $\bq \in \Delta(\cJ)$
an element $\phi(\bq) \in \Delta(\cI)$ such that
\[
\phi(\bq) \in \mathop{\mathrm{argmax}}_{\bp \in \Delta(\cI)} \rho \bigl( \bp, \, \tH(\bq) \bigr)\,.
\]
Then, given any $\bq \in \Delta(\cJ)$, we note that for all $\bq'$ satisfying
$\tH(\bq') = \tH(\bq)$, we have by definition of $\rho$,
\[
r \bigl( \phi(\bq), \, \bq' \bigr) \geq \rho \bigl( \phi(\bq), \, \tH(\bq') \bigr)
= \max_{\bp \in \Delta(\cI)} \rho \bigl( \bp, \, \tH(\bq') \bigr)\,,
\]
which shows that $\underline{r} \bigl( \phi(\bq), \, \bq' \bigr) \in \cC$.
The required condition is thus satisfied.

We then show that Assumption~\ref{As:BPL} is satisfied. To do so, we will actually prove the
stronger property that the mappings $\om(\,\cdot\,,\sigma)$ are piecewise linear for all $\sigma \in \cF$; we fix such a $\sigma$
in the sequel. Only the first coordinate $r$ of $\underline{r}$ depends on $\bp$, so the desired property is true
if and only if the mapping $\om_1(\,\cdot\,,\sigma)$ defined by
\[
\bp \in \Delta(\cI) \,\, \longmapsto \,\, \om_1(\bp,\sigma) = \left\{ r(\bp,\bq) : \ \ \bq \in \Delta(\cJ) \
\mbox{such that} \ \tH(\bq) = \sigma \right\}
\]
is piecewise linear.
Since $\tH$ is linear, the set
\[
\left\{ \bq \in \Delta(\cJ) \ \mbox{such that} \ \tH(\bq) = \sigma \, \right\}
\]
is a polytope, thus, the convex hull of some finite set $\{\bq_{\sigma,1},\,\ldots,\,\bq_{\sigma,M}\} \subset \Delta(\cJ)$.
Therefore, for every $\bp \in \Delta(\cI)$, by linearity of $r$ (and by the fact that it takes one-dimensional values),
\begin{equation}
\label{eq:omlin}
\om_1(\bp,\sigma) = \co \Big\{ r(\bp,\bq_{\sigma,1}), \, \ldots, \, r(\bp,\bq_{\sigma,M}) \Big\}
= \left[\min_{k\in\{1,..,M\}} r(\bp,\bq_{\sigma,k})\,, \,\,\, \max_{k'\in\{1,..,M\}} r(\bp,\bq_{\sigma,k'}) \right],
\end{equation}
where $\co$ stands for the convex hull. Since all applications
$r(\,\cdot\,,\bq_{\sigma,k})$ are linear, their minimum and their maximum are piecewise linear functions,
thus $\om_1(\,\cdot\,,\sigma)$ is also piecewise linear.
Assumption~\ref{As:BPL} is thus satisfied, as claimed.
\medskip

Theorem~\ref{TheoGeneralCase} can therefore be applied to
exhibit the convergence rates; we simply need to relate the quantity of interest
here to the one considered therein. To that end we use the fact that the mapping
\[
\sigma \in \cF \, \longmapsto \, \max_{\bp \in \Delta(\cI)} \, \rho(\bp,\sigma)
\]
is Lipschitz, with Lipschitz constant in $\ell^2$--norm denoted by $L_\rho$; the proof of this fact
is detailed below.

Now, the regret is non positive as soon as $\sum_{t=1}^T \underline{r}(I_t,J_t)/T$ belongs to $\cC$;
we therefore only need to consider the case when this average is not in $\cC$. In the latter case, we
denote by $(\widetilde{r}_T,\widetilde{\sigma}_T)$ its projection in $\ell^2$--norm onto $\cC$.
We have first
that the defining inequality of $\cC$ is an equality on its border, so that
\[
\widetilde{r}_T = \max_{\bp \in \Delta(\cI)} \,\, \rho \bigl( \bp, \widetilde{\sigma}_T \bigr)\,;
\]
and second, that
\begin{eqnarray*}
R^{\ext}_T & =&
\max_{\bp \in \Delta(\cI)} \,\, \rho \Bigl( \bp, \tH \bigl( \wh{\bq}_T \bigr) \Bigr) -
\frac{1}{T} \sum_{t=1}^T r(I_t,J_t) \\
&\leq& \left|\max_{\bp \in \Delta(\cI)} \,\, \rho \Bigl( \bp, \tH \bigl( \wh{\bq}_T \bigr) \Bigr)- \max_{\bp \in \Delta(\cI)} \,\, \rho \bigl( \bp,
\widetilde{\sigma}_T \bigr)\right|+ \left|\, \widetilde{r}_T -
\frac{1}{T} \sum_{t=1}^T r(I_t,J_t) \right|\\
&\leq& L_\rho \,  \Bnorm[ \widetilde{\sigma}_T - \tH \bigl( \wh{\bq}_T \bigr)]_2 + \left|\, \widetilde{r}_T -
\frac{1}{T} \sum_{t=1}^T r(I_t,J_t) \right|\\
& \leq & \sqrt{2} \, \max \bigl\{ L_\rho, 1 \bigr\} \, \norm[ {\left[ \begin{array}{c} \widetilde{r}_T \\ \widetilde{\sigma}_T \end{array} \right]}
- \frac{1}{T} \sum_{t=1}^T \underline{r}(I_t,J_t)]_2 \\
& = & \sqrt{2} \, \max \bigl\{ L_\rho, 1 \bigr\} \, \inf_{c \in \cC} \ \norm[c - \frac{1}{T} \sum_{t=1}^T \underline{r}(I_t,J_t)]_2.
\end{eqnarray*}
The claimed rates are now seen to follow from the ones indicated in Theorem~\ref{TheoGeneralCaseVariant}.

It only remains to prove the indicated Lipschitzness.
(All Lipschitzness statements that follow will be with respect to the $\ell^2$--norms.)
We have by Definition~\ref{def:oom} that for all $\bp \in \Delta(\cI)$ and $\sigma \in \cF$,
\[
\rho(\bp,\sigma) = \min \,\, \oom_1\bigl(\bp,\Phi(\sigma)\bigr)\,,
\]
where the linear $\oom_1$ is indifferently either relative to $\om_1$ or is the projection onto the first component
of the function $\oom$ relative to $\om$.
By Remark~\ref{rk:Lip} the mapping $\sigma \in \cF \mapsto \Phi(\sigma)$ is $\kappa_{\Phi}$--Lipschitz;
this entails, by Lemma~\ref{lm:DVT}, that for all $\bp \in \Delta(\cI)$,
the mapping $\sigma \in \cF \mapsto \rho(\bp,\sigma)$ is $R\sqrt{N_\cB}\,\, \kappa_\Phi$--Lipschitz.
In particular, since the latter Lipschitz constant is independent of $\bp$, the mapping
\[
\sigma \in \cF \,\, \longmapsto \,\, \max_{\bp \in \Delta(\cI)} \rho(\bp,\sigma)
\]	
is $R\sqrt{N_\cB}\,\, \kappa_\Phi$--Lipschitz as well, which concludes the proof.
\end{proof}

A similar argument to the one in \citemor{Per11JMLR} shows that the convex set $\cC$ is defined by a finite number of piecewise linear equations, it is therefore a polyhedron so the projection onto it, and as well the computation  of the strategy, can be done efficiently. We sketch the argument below, and refer the reader to  \citemor{Per11JMLR} for details.
Equation~\eqref{eq:omlin} indicates a priori that for each $\sigma \in \cF$,
there exist a finite number $M_\sigma$ (depending on $\sigma$) of mixed actions $\bq_{\sigma,1}, \, \ldots, \, \bq_{\sigma,M_\sigma}$
such that for all $\bp \in \Delta(\cI)$, we have $\rho(\bp,\sigma)=\min \bigl\{ r(\bp,\bq_{\sigma,1}),\ldots, r(\bp,\bq_{\sigma,M_\sigma})
\bigr\}$. But by an argument stated in \citemor{Per11JMLR},
\[
\sigma \, \longmapsto \, \left\{ \bq \in \Delta(\cJ) \ \mbox{such that} \ \tH(\bq) = \sigma \, \right\}
\]
evolves in a piecewise linear way and thus
there exist a finite number $M$ of piecewise linear functions
$\sigma \mapsto q'_{\sigma,k}$, with $k = 1,\ldots,M$, such that, for all $\sigma \in \cF$,
\[
\bigl\{ \bq_{\sigma,1}, \, \ldots, \, \bq_{\sigma,M_\sigma} \bigr\}
= \bigl\{ \bq'_{\sigma,1}, \, \ldots, \, \bq'_{\sigma,M} \bigr\}\,.
\]
(There can be some redundancies between the $\bq'_{\sigma,k}$.)
Because of this, we have that for all $\bp \in \Delta(\cI)$ and $\sigma \in \cF$,
\[
\rho(\bp,\sigma) =\min \bigl\{ r(\bp,\bq'_{\sigma,1}),\ldots, r(\bp,\bq'_{\sigma,M}) \bigr\}\,.
\]
Each function $\sigma \mapsto \bq'_{\sigma,k}$ being piecewise linear,
one can construct a finite set $\{\bp_1,\ldots,\bp_K\} \subset \Delta(\cI)$ such that,
for any $\sigma \in \cF$, the mapping $\bp \mapsto \rho(\bp,\sigma)$ is maximized at one of these $\bp_k$.
The convex set $\cC$ is therefore defined by a finite number of piecewise linear
equations, it is therefore a polyhedron; therefore the projection onto it,
hence the computation of the proposed strategy, can be done efficiently.

\subsubsection{Internal / swap regret.}

\citemor{FoVo99} defined internal regret with full monitoring as follows. A player has no internal regret if, for every action $i \in \cI$,
he has no external regret on the stages when this specific action $i$ was played. In other words,  $i$ is the best response to the empirical distribution of action of the other player on these stages.

With partial monitoring, the first player evaluates his payoffs in a pessimistic way
through the function $\rho$ defined above. This function is not linear over $\Delta(\cI)$ in general (it is concave),
so that the best responses are not necessarily pure actions $i \in \cI$ but mixed actions, i.e., elements of $\Delta(\cI)$.
Following \citemor{LeSo07} one therefore can partition the stages not depending on the pure actions actually played
but on the mixed actions $\bp_t \in \Delta(\cI)$ used to draw them. To this end, it is convenient to
assume that the strategies of the first player need to pick these mixed actions in a finite (but possibly thin) grid
of $\Delta(\cI)$, which we denote by $\bigl\{ \bp_g, \,\, g \in \cG \bigr\}$, where $\cG$ is a finite set.
At each round $t$, the first player picks an index $G_t \in \cG$ and uses the distribution $\bp_{G_t}$ to draw his action $I_t$.
Up to a standard concentration-of-the-measure argument, we will measure the payoff at round $t$ with $r\bigl( \bp_{G_t}, J_t \bigr)$
rather than with $r( I_t, J_t )$.

For each $g \in \cG$, we denote by $N_T(g)$ the number of stages in $\{ 1,\ldots,T \}$ for which we had $G_t = g$
and, whenever $N_T(g) > 0$,
\[
\wh{\bq}_{T,g} = \frac{1}{N_T(g)} \sum_{t : G_t = g} \, \delta_{J_t}\,.
\]
We define $\wh{\bq}_{T,g}$ is an arbitrary way when $N_T(g) = 0$.
The internal regret of the first player at round $T$ is measured as
\[
R_T^{\inter} = \max_{g,g' \in \cG} \frac{N_T(g)}{T} \biggl( \rho \Bigl( \bp_{g'}, \tH \bigl( \wh{\bq}_{T,g} \bigr) \Bigr)
- r \bigl( \bp_{g}, \wh{\bq}_{T,g} \bigr) \biggr)\,.
\]
Actually, our proof technique rather leads to the minimization of some swap regret (see \citemor{BlMa07}
for the definition of swap regret in full monitoring):
\[
R_T^{\swap} = \sum_{g \in \cG} \frac{N_T(g)}{T} \, \biggl( \max_{g' \in \cG} \rho \Bigl( \bp_{g'}, \tH \bigl( \wh{\bq}_{T,g} \bigr) \Bigr)
- r \bigl( \bp_{g}, \wh{\bq}_{T,g} \bigr) \biggr)_+ \, .
\]

Again, the following bound on the swap regret easily follows from Theorem~\ref{TheoGeneralCase};
the latter constructs a simple and direct strategy to control the swap regret, thus also the internal regret.
It therefore improves the results of \citemor{LeSo07} and \citemor{Per09}, two articles that presented
more involved and less efficient strategies to do so (strategies
based on auxiliary strategies using grids that need to be refined over time and whose complexities is exponential in the size
of these grids; ideas all in all similar to what is done in calibration, see the references
provided in Section~\ref{sec:robappr-gal}). Moreover, we provide convergence rates.

\begin{corollary}
\label{cor:IR}
The first player has an explicit
strategy such that for all $T$ and all strategies of the second player, with probability
at least $1-\delta$,
\[
R^{\swap}_T \,\, \leq \,\,
\square \left(T^{-1/5}\sqrt{\ln \frac{T}{\delta}} + T^{-2/5} \ln \frac{T}{\delta} \right)
\]
for some constant $\square$ depending only on the game $(r,\,H)$ at hand and on the size of the finite
grid $\cG$.
\end{corollary}


\begin{proof}
The proof of this corollary is based on ideas similar to the ones used in the proof of
Corollary~\ref{cor:ER}; $\cG$ will play the role of the action set of the first player.
The proof proceeds in four steps. In the first step, we construct an approachability setup and show that Condition~{\apm}
applies. In the second step, we show that Assumption~\ref{As:BPL} is satisfied.
In the third step we analyze the convergence rates of the swap regret.
In the fourth and final step, we show that the set we are approaching possess some smoothness properties by
providing a uniform Lipschitz bound on certain functions.

{\bf Step 1:}
We denote by
\[
\cF_{\cone} = \bigl\{ \lambda \sigma, \ \ \sigma \in \cF, \ \lambda \in \R_+ \bigr\}
\]
the cone generated by $\cF$ and extend linearly
$\rho : \Delta(\cI) \times \cF \to \R$ into a mapping $\rho : \Delta(\cI) \times \cF_{\cone} \to R$
as follows: for all $\bp \in \Delta(\cI)$, for all $\lambda \geq 0$ with $\lambda \ne 1$, and all $\sigma \in \cF$,
\[
\rho(\bp,\,\lambda\sigma) = \left\{
\begin{array}{lp{0.5cm}l}
0 & & \mbox{if} \ \lambda = 0, \\
\lambda \, \rho( \bp, \, \sigma) & &
\mbox{if} \ \lambda > 0.
\end{array}
\right.
\]
In the sequel, we embed $\cF_{\cone}$ into $\R^{\cI \times \cH}$.

The closed convex set $\cC$ and the vector-valued payoff
function $\underline{r}$ are then respectively defined by
\[
\cC = \left\{ (z_g, \bv_g)_{g \in \cG} \in \bigl( \R \times \cF_{\cone} \bigr)^{\cG} : \ \
\forall \, g \in \cG, \ \ z_g \geq \max_{g' \in \cG} \rho \bigl( \bp_{g'}, \bv_g \bigr) \right\}
\]
and, for all $(g,j) \in \cG \times \cJ$,
\[
\underline{r}(g,j) = \left[ \begin{array}{c} r\bigl( \bp_g, j \bigr) \, \ind_{ \{ g' = g \} } \vspace{.2cm} \\
\tH(\delta_j) \, \ind_{ \{ g' = g \} } \end{array} \right]_{g' \in \cG} \,.
\]
To show that $\cC$ is $\underline{r}$--approachable,
we associate with each $\bq \in \Delta(\cJ)$
an element $g^\star(\bq) \in \cG$ such that
\[
g^\star(\bq) \in \mathop{\mathrm{argmax}}_{g \in \cG} \rho \bigl( \bp_g, \, \tH(\bq) \bigr)\,.
\]
Then, given any $\bq \in \Delta(\cJ)$, we note that for all $\bq'$ satisfying
$\tH(\bq') = \tH(\bq)$, the components of the vector $\underline{r}\bigl( g^\star(\bq), \bq' \bigr)$
are all null but the ones corresponding to $g^\star(\bq)$, for which we have
\[
r\bigl( \bp_{g^\star(\bq)}, \bq'\bigr)
\geq \rho \Bigl( \bp_{g^\star(\bq)}, \tH \bigl( \bq' \bigr) \Bigr)
= \rho \Bigl( \bp_{g^\star(\bq)}, \tH \bigl( \bq \bigr) \Bigr)
= \max_{g' \in \cG} \rho \Bigl( \bp_{g'}, \tH \bigl( \bq \bigr) \Bigr)
= \max_{g' \in \cG} \rho \Bigl( \bp_{g'}, \tH \bigl( \bq' \bigr) \Bigr)\,,
\]
where the first inequality is by definition of $\rho$. Therefore,
$\underline{r}\bigl( g^\star(\bq), \bq' \bigr) \in \cC$.
Condition~{\apm} in Lemma~\ref{lm:Psi} and Theorem~\ref{TheoGeneralCase}
is thus satisfied, so that we have approachability.

{\bf Step 2:} We then show that Assumption~\ref{As:BPL} is satisfied.
It suffices to show that for all $\sigma \in \cF$,
the application
\[
\pi = (\pi_g)_{g \in \cG} \in \Delta(\cG) \ \longmapsto \ \om_1(\pi,\sigma) = \Bigl\{
\bigl( \pi_g \, r(\bp_g,\bq) \bigr)_{g \in \cG} :
\ \ \bq \in \Delta(\cJ) \ \mbox{such that} \ \tH(\bq) = \sigma \Bigr\}
\]
is piecewise linear (as the other components in the definition of $\om$ are linear in $\pi$).
This is the case since for each $g$, the application
\[
\pi \in \Delta(\cG) \ \longmapsto \ \Bigl\{
\pi_g \, r(\bp_g,\bq) :
\ \ \bq \in \Delta(\cJ) \ \mbox{such that} \ \tH(\bq) = \sigma \Bigr\}
\]
is seen to be piecewise linear, by using the same one-dimensional argument as in the
proof of Corollary~\ref{cor:ER}.

{\bf Step 3:}
We now exhibit the convergence rates.
In view of the form of the defining set of constraints for $\cC$,
the coordinates of the elements in $\cC$ can be grouped according to each $g \in \cG$ and
projections onto $\cC$ can therefore be done separately for each such group.
The group $g$ of coordinates of $\sum_{t=1}^T \underline{r}(G_t,J_t)/T$
is formed by
\[
\frac{N_T(g)}{T} \, r\bigl(\bp_g, \wh{\bq}_{T,g} \bigr) \qquad \mbox{and} \qquad
\frac{N_T(g)}{T} \, \tH\bigl(\wh{\bq}_{T,g} \bigr)\,;
\]
when
\[
\frac{N_T(g)}{T} \, r\bigl(\bp_g, \wh{\bq}_{T,g} \bigr) \geq \max_{g' \in \cG} \,\, \rho \!
\left( \bp_{g'}, \, \frac{N_T(g)}{T} \, \tH\bigl(\wh{\bq}_{T,g} \bigr) \right),
\]
we denote these quantities by $\tr_{T,g}$ and $\wt{\bv}_{T,g}$. Otherwise, we project
this pair on the set
\[
\cC_g = \left\{ (z_g, \bv_g) \in \R \times \cF_{\cone} : \ \
z_g \geq \max_{g' \in \cG} \rho \bigl( \bp_{g'}, \bv_g \bigr) \right\}
\]
and denote by $\tr_{T,g}$ and $\wt{\bv}_{T,g}$ the coordinates of the projection;
they satisfy the defining inequality of $\cC_g$ with equality,
\[
\tr_{T,g} = \max_{g' \in \cG} \rho \bigl( \bp_{g'}, \wt{\bv}_g \bigr)\,.
\]

By distinguishing for each $g$ according to which of the two cases above arose
(for the first inequality), we may decompose and upper bound the swap regret as follows,
\begin{eqnarray*}
R^{\swap}_T & =&
\sum_{g \in \cG} \frac{N_T(g)}{T} \, \biggl( \max_{g' \in \cG} \rho \Bigl( \bp_{g'}, \tH \bigl( \wh{\bq}_{T,g} \bigr) \Bigr)
- r \bigl( \bp_{g}, \wh{\bq}_{T,g} \bigr) \biggr)_+ \\
& = &
\sum_{g \in \cG} \left( \max_{g' \in \cG} \rho \!
\left( \bp_{g'}, \, \frac{N_T(g)}{T} \, \tH\bigl(\wh{\bq}_{T,g} \bigr) \right)
- \frac{N_T(g)}{T} \, r\bigl(\bp_g, \wh{\bq}_{T,g} \bigr) \right)_{\! +} \\
& \leq &
\sum_{g \in \cG} \left| \max_{g' \in \cG} \rho \!
\left( \bp_{g'}, \, \frac{N_T(g)}{T} \, \tH\bigl(\wh{\bq}_{T,g} \bigr) \right)
- \max_{g' \in \cG} \rho \bigl( \bp_{g'}, \wt{\bv}_{g,T} \bigr) \right| + \sum_{g \in \cG} \, \biggl| \wt{r}_{T,g}
- \frac{N_T(g)}{T} \, r\bigl(\bp_g, \wh{\bq}_{T,g} \bigr) \biggr| \\
& \leq & \sum_{g \in \cG} \overline{L}_\rho \norm[\frac{N_T(g)}{T} \, \tH\bigl(\wh{\bq}_{T,g} \bigr) - \wt{\bv}_{g,T}]_2
+ \sum_{g \in \cG} \, \biggl| \wt{r}_{T,g} - \frac{N_T(g)}{T} \, r\bigl(\bp_g, \wh{\bq}_{T,g} \bigr) \biggr|\,,
\end{eqnarray*}
where we used a fact proved below, that the application
\[
\bv \in \cF_{\cone} \, \longmapsto \, \max_{g' \in \cG} \, \rho\bigl(\bp_{g'},\bv\bigr)
\]
is $\overline{L}_\rho$--Lipschitz. In the last inequality we had a sum of $\ell^2$--norms, which
can be bounded by a single $\ell^2$--norm,
\begin{eqnarray*}
R^{\swap}_T & \leq & \max\bigl\{\overline{L}_\rho,1\bigr\} \, \sqrt{2N_{\cG}}  \, \left\Arrowvert
\left[ \begin{array}{c} \widetilde{r}_{T,g} \vspace{.1cm} \\
\widetilde{\bv}_{T,g} \end{array} \right]_{g \in \cG} - \frac{1}{T} \sum_{t=1}^T \underline{r}(I_t,J_t)
\right\Arrowvert_2 \\
& \leq & \max\bigl\{\overline{L}_\rho,1\bigr\}
\, \sqrt{2N_{\cG}}  \, \inf_{c \in \cC} \norm[ c - \frac{1}{T} \sum_{t=1}^T \underline{r}(I_t,J_t) ]_2\,,
\end{eqnarray*}
where we denoted by $N_\cG$ the cardinality of $\cG$.
Resorting to the convergence rate stated in Theorem~\ref{TheoGeneralCaseVariant} concludes the proof, up to
the claimed Lipschitzness, which we now prove. (All Lipschitzness statements that follow will be with
respect to the $\ell^2$--norms.)

{\bf Step 4:}
To do so, it suffices to show that for all fixed elements $\bp \in \Delta(\cI)$,
the functions $\bv \in \cF_{\cone} \mapsto \rho(\bp,\bv)$
are Lipschitz, with a Lipschitz constant $\overline{L}_\rho$ that is independent of $\bp$.
Note that we already proved at the end of the proof of Corollary~\ref{cor:ER} that
$\sigma \in \cF \mapsto \rho(\bp,\sigma)$ is Lipschitz, with a Lipschitz constant $L_\rho$
independent of $\bp$. Consider now two elements $\bv, \, \bv' \in \cF_{\cone}$,
which we write as $\bv = \lambda \sigma$ and $\bv' = \lambda' \sigma'$,
with $\sigma, \, \sigma' \in \cF$ and $\lambda, \, \lambda' \in \R_+$.
Using triangle inequalities, the Lipschitzness of $\rho$ on $\cF$,
and the fact that $r$ thus $\rho$ are bounded by $R$,
\begin{eqnarray*}
\bigl| \rho(\bp,\lambda\sigma) - \rho(\bp,\lambda'\sigma') \bigr| & \leq &
\bigl| \lambda \bigl( \rho(\bp,\sigma) - \rho(\bp,\sigma') \bigr) \bigr|
+ \bigl| (\lambda - \lambda') \rho(\bp,\sigma') \bigr| \\
& \leq & \lambda \, L_\rho \bnorm[\sigma - \sigma']_2 +
R \, \bigl| \lambda - \lambda' \bigr| \\
& \leq & L_\rho \bnorm[\lambda\sigma - \lambda'\sigma' + (\lambda' - \lambda)\sigma' ]_2 +
R \, \bigl| \lambda - \lambda' \bigr| \\
& \leq & L_\rho \bnorm[\lambda\sigma - \lambda'\sigma']_2 +
\bigl( R + L_\rho N_{\cI} \bigr) \, \bigl| \lambda - \lambda' \bigr|\,,
\end{eqnarray*}
where we used also for the last inequality that since $\sigma$ is a vector of $N_{\cI}$ probability
distributions over the signals, $\norm[\sigma]_2 \leq \norm[\sigma]_1 = N_{\cI}$.
To conclude the argument, we simply need to show that $\bigl| \lambda - \lambda' \bigr|$
can be bounded by $\bnorm[\lambda\sigma - \lambda'\sigma']_2$ up to some universal constant,
which we do now. We resort again to the fact that
$\norm[\sigma]_1 = \norm[\sigma']_1 = N_{\cI}$ and can thus write, thanks to a triangle
inequality and assuming with no loss
of generality that $\lambda' < \lambda$, that
\[
\bigl| \lambda - \lambda' \bigr| = \frac{1}{N_{\cI}} \Bigl( \lambda \norm[\sigma]_1 - \lambda' \norm[\sigma']_1 \Bigr)
\leq \frac{1}{N_{\cI}} \norm[ \lambda \sigma - \lambda' \sigma' ]_1
\leq \frac{\sqrt{N_{\cH} N_{\cI}}}{N_{\cI}} \norm[ \lambda \sigma - \lambda' \sigma' ]_2\,,
\]
where we used the Cauchy-Schwarz inequality for the final step.
One can thus take, for instance,
\[
\overline{L}_\rho = L_\rho + \bigl( R + L_\rho N_{\cI} \bigr) \sqrt{\frac{N_{\cH}}{N_{\cI}}}\,.
\]
This concludes the proof.
\end{proof}

\section{Approachability in the case of general games with partial monitoring.}
\label{sec:galstruc}

Unfortunately, as is illustrated in the following example,
there exist games with partial monitoring that are not bi-piecewise linear.

\begin{example}
The following game (with the same action and signal sets as in Example~\ref{ex:1}) is not bi-piecewise linear.
\begin{center}
\begin{tabular}{ccccc}
\hline
& & $L$ & $M$ & $R$ \\
\hline
$T$ & & $(1,0,0,0)$ / $\sigdark$ & $(0,0,1,0)$ / $\sigdark$ & $(2,0,4,0)$ / $\sigother$ \\
$B$ & & $(0,1,0,0)$ / $\sigdark$ & $(0,0,0,1)$ / $\sigdark$ & $(0,3,0,5)$ / $\sigother$ \\
\hline
\end{tabular}
\end{center}
\end{example}

\begin{proof}
We denote mixed actions of the first player by $(p,1-p)$, where $p \in [0,1]$
denotes the probability of playing $T$ and $1-p$ is the probability of playing $B$.
It is immediate that $\om\bigl( (p,1-p), \, \sigdark \bigr)$ can be identified with the
set of all product distributions on $2 \times 2$ elements with first marginal distribution $(p,1-p)$.
The proof of Lemma~\ref{lm:geocorr} shows that the set $\cB$ associated with any game
always contains the Dirac masses on each signal; that is, $\delta_\sigdark \in \cB$. But
for $p \ne p'$ and $\lambda \in (0,1)$,
denoting $\overline{p} = \lambda \, p + (1-\lambda) p'$, one necessarily has that
\[
\om\bigl( (\overline{p},1-\overline{p}), \, \sigdark \bigr) \, \varsubsetneq \,
\lambda \, \om\bigl( (p,1-p), \, \sigdark \bigr) + (1-\lambda) \,
\om\bigl( (p',1-p'), \, \sigdark \bigr)\,;
\]
the inclusion $\subseteq$ holds by concavity of $\om$ in its first argument (Lemma~\ref{lm:cvxconc})
but this inclusion is always strict here
since the left-hand side is formed by product distributions while the right-hand side
also contains distributions with correlations. Hence, bi-piecewise linearity
cannot hold for this game.
\end{proof}

However, we will show that if Condition~{\apm} holds there exist strategies with a constant
per-round complexity to approach polytopes even when the game is not bi-piecewise linear.
That is, by considering simpler closed convex sets $\cC$, no assumption is needed on the pair $(r,H)$.

We will conclude this section by indicating that thanks to a doubling trick, Condition~{\apm} is still
 sufficient for approachability in the most general case when no assumption is made neither
on $(r,H)$ nor on $\cC$---at the cost, however, of inefficiency.

\subsection{Approachability of the negative orthant in the case of general games.}

For the sake of simplicity, we start with the case of the negative orthant $\R^d_-$.
Our argument will be based on Lemma~\ref{lm:geocorr}; we use in the sequel the
objects and notation introduced therein. We denote by $r = (r_k)_{1 \leq k \leq d}$
the components of the $d$--dimensional payoff function $r$ and introduce, for all
$k \in \{ 1,\ldots,d\}$, the set-valued mapping $\tm_{k}$ defined by
\[
\tm_k : \ \ (\bp,\,b) \in \Delta(\cI) \times \cB \ \longmapsto \ \tm_k(\bp,b) =
\Bigl\{ r_k(\bp,\bq) : \ \ \bq \in \Delta(\cJ) \ \mbox{such that} \ \tH(\bq) = b \Bigr\}\,.
\]
The mapping $\tm$ is then defined as the Cartesian product of the $\tm_k$; formally,
for all $\bp \in \Delta(\cI)$ and $b \in \cB$,
\[
\tm(\bp,b) = \Bigl\{ (z_1,\ldots,z_d) : \qquad \forall k \in \{ 1,\ldots,d\}, \quad z_k \in  \tm_k(\bp,b)
\Bigr\}\,.
\]
We then linearly extend this mapping into a set-valued mapping $\tm$ defined on $\Delta(\cI) \times \Delta(\cB)$
and finally consider the set-valued mapping $\breve{m}$ defined on $\Delta(\cI) \times \cF$ by
\[
\forall \, \sigma \in \cF, \quad \forall \, \bp \in \Delta(\cI), \qquad
\breve{m}(\bp,\sigma) = \tm\bigl(\bp,\Phi(\sigma) \bigr) = \sum_{b \in \cB} \Phi_b(\sigma) \, \tm(\bp,b)\,,
\]
where $\Phi$ refers to the mapping defined in Lemma~\ref{lm:geocorr} (based on $\om$). The lemma below indicates why $\breve{m}$
is an excellent substitute to $\om$ in the case of the approachability of the orthant $\R^d_-$.

\begin{lemma}
\label{lm:proxy}
The set-valued mappings $\breve{m}$ and $\om$ satisfy that
for all $p \in \Delta(\cI)$ and $\sigma \in \cF$,
\begin{enumerate}
\item the inclusion
$\om(\bp,\sigma) \subseteq \breve{m}(\bp, \sigma)$ holds;
\item if $\om(\bp,\sigma) \subseteq \R_-^d$, then one also has
$\breve{m}(\bp,\sigma) \subseteq \R_-^d$.
\end{enumerate}
\end{lemma}

The interpretations of these two properties are that: {1.} $\breve{m}$--robust approaching
a set $\cC$ is more difficult than $\om$--robust approaching it; and {2.} that
if Condition~{\apm} holds for $\om$ and $\R_-^d$, it also holds for $\breve{m}$ and $\R_-^d$. \\

\begin{proof}
For  property {1.}, note that by the component-wise construction of $\tm$,
\[
\forall \, b \in \cB, \quad \forall \, \bp \in \Delta(\cI), \qquad
\om(\bp,b) \subseteq \tm(\bp,b)\,;
\]
Lemma~\ref{lm:geocorr}, the linear extension of $\tm$, and the definition of $\breve{m}$
then show that
\[
\forall \, \sigma \in \cF, \quad \forall \, \bp \in \Delta(\cI), \qquad \quad
\om(\bp,\sigma) = \sum_{b \in \cB} \Phi_b(\sigma) \, \om(\bp,b)
\subseteq \tm\bigl(\bp, \, \Phi(\sigma) \bigr) = \breve{m}(\bp,\sigma)\,.
\]

As for property {2.}, it suffices to work component-wise. Note that
(by Lemma~\ref{lm:geocorr} again) the stated assumption exactly means that
$\sum_{b \in \cB}\Phi_b(\sigma) \, \om(\bp,b) \subset \R_-^d$. In particular, rewriting the non-positivity constraint
for each of the $d$ components of the payoff vectors, we get
\[
\sum_{b \in \cB} \Phi_b(\sigma) \, \tm_k(\bp,b) \subseteq \R_-\,,
\]
for all $k \in \{ 1,\ldots,d\}$; thus, in particular, $\sum_{b \in \cB} \Phi_b(\sigma) \, \tm(\bp,b)
= \breve{m}(\bp,\sigma) \subseteq \R_-^d$.
\end{proof}

We can then extend the result of the previous section without the bi-piecewise linearity
assumption.

\begin{theorem}
\label{th:orthant}
If Condition~{\apm} is satisfied for $\om$ and $\R_-^d$, then there exists a strategy for $(r,H)$--approaching
$\R_-^d$ at a rate of the order of $T^{-1/5}$, with a constant per-round complexity.
\end{theorem}

\begin{proof}
The assumption of the theorem and Property 2.\ of Lemma~\ref{lm:proxy} imply
that Condition~{\apm} holds for $\R_-^d$ and $\breve{m}$; furthermore, the latter corresponds to a bi-piecewise linear game
as can be seen by noting, similarly to what was done in the section devoted to regret minimization (Section~\ref{se:apptoregretmin}),
that each $\tm_k$, being based on the scalar payoff function $r_k$, is a piecewise linear function.
Thus, $\breve{m}$ is also a piecewise linear function.

Therefore, the steps between Equations~\eqref{eq:term3}--\eqref{eq:term5} of the proof of Theorem~\ref{TheoGeneralCase}
(or the corresponding statements in the proof of Theorem~\ref{TheoGeneralCaseVariant})
can be adapted by replacing $\om$ and $\oom$ by, respectively, $\tm$, $\breve{m}$, and its extension corresponding to
Definition~\ref{def:oom}. The result follows.
\end{proof}

\subsection{Approachability of polytopes in the case of general games.}

If that the target set $\cC$ is a polytope, then $\cC$ can be written  as the intersection of a finite number of half-planes, i.e., there exists a finite family $\bigl\{ (e_k,\,f_k) \in \R^d \times \R, \ k \in \cK \bigr\}$ such that
\[
\cC= \bigl\{ z \in \R^d : \quad \langle z, e_k \rangle \leq f_k, \ \ \forall \, k \in \cK \bigr\}.
\]
Given the original (not necessarily bi-piecewise linear) game $(r,H)$, we introduce another
game $(r_\cC,H)$, whose payoff function $r_\cC : \cI \times \cJ \to \R^{\cK}$
is defined as
\[
\forall \, i \in \cI, \quad \forall \, j \in \cJ, \qquad
r_\cC(i,j)= \Big[\langle r(i,j), e_k \rangle - f_k  \Big]_{k \in \cK}.
\]
The following lemma follows by rewriting the above.

\begin{lemma}
Given a polytope $\cC$,
the $(r,H)$--approachability of $\cC$ and the $\bigl(r_\cC,H \bigr)$--approachability of $\R_-^d$ are equivalent in the sense that
every strategy for one problem translates to a strategy for the other problem.
In addition, Condition~{\apm} holds for $(r,H)$ and $\cC$ if and only if it holds for $\bigl(r_\cC,H \bigr)$ and $\R_-^d$.
\end{lemma}

Via the lemma above,
Theorem~\ref{th:orthant} indicates that Condition~{\apm} for $(r,H)$ and $\cC$ is a sufficient
condition for the $(r,H)$--approachability of $\cC$ and provides a strategy to do so. (The
per-round complexity of this strategy depends in particular at least linearly on the cardinality of $\cK$.)

\subsection{Approachability of general convex sets in the case of general games.}

A general closed convex set can always be approximated arbitrarily well by a polytope (where the number of
vertices of the latter however increases as the quality of the approximation does). Therefore,
via playing in regimes, Condition~{\apm} is also seen to be sufficient to $(r,H)$--approach any general closed
convex set $\cC$. However, the computational complexity of the resulting strategy is much larger:
the per-round complexity increases over time (as the numbers of vertices of the approximating polytopes do).

\appendix

\section{An auxiliary result of calibration.}
\newcommand{\uz}{\underline{0}}
\newcommand{\uo}{\underline{1}}

{$\,$}

We prove here~\eqref{def:etacal} for a given $\eta > 0$ and do so by following the methodology
of~\citemor{MaSt10}. (Note that this result is of independent interest.)

We actually assume that the covering $\by^1,\ldots,\by^{N_{\eta}}$
is slightly finer than what was required around~\eqref{def:etacal}
and that it forms an $\eta/N_\cB$--grid of $\Delta(\cB)$,
i.e., that for all $\by \in \Delta(\cB)$, there exists $\ell \in \{ 1,\ldots,N_{\eta} \}$
such that $\norm[\by-\by^\ell]_1 \leq \eta/N_\cB$.

We recall that elements $\by \in \cB$
are denoted by $\by = (y_b)_{b \in \cB}$ and we identify $\Delta(\cB)$ with a subset
of $\R^{N_\cB}$. In particular, $\mathbb{I}_b$, the Dirac mass on a given $b \in \cB$, is a binary vector
whose only non-null component is the one indexed by $b$.
Finally, we denote by
\[
\uz = (0,\ldots,0) \qquad \mbox{and} \qquad \uo = (1,\ldots,1)
\]
the elements of $\R^{\cB}$ respectively formed by zeros and ones only.

We consider a vector-valued payoff function $C : \{ 1,\ldots,N_{\eta} \} \times \cB \to \R^{2 N_\eta N_\cB}$
defined as follows; for all $\ell \in \{ 1,\ldots,N_{\eta} \}$ and for all $b \in \cB$,
\[
C(\ell,b) = \left( \uz, \ \ldots, \ \uz, \ \ \by^\ell-\ind_b - \frac{\eta}{N_{\cB}} \uo, \ \
\ind_b-\by^\ell - \frac{\eta}{N_{\cB}} \uo, \ \ \uz, \ \ldots, \ \uz \right),
\]
which is a vector of $2N_{\eta}$ elements of $\R^{\cB}$ composed by $2(N_{\eta}-1)$
occurrences of the zero element $\uz \in \R^{\cB}$ and two non-zero elements, located in the positions
indexed by $2\ell-1$ and $2\ell$.

We now show that the closed convex set $(\R_-)^{2 N_{\eta} N_{\cB}}$ is $C$--approachable;
to do so, we resort to the characterization stated in Theorem~\ref{th:appr}. To each $\by \in \Delta(\cB)$
we will associate a pure action $\ell_{\by}$ in $\{1,\ldots,N_\eta\}$ so that $C\bigl( \ell_{\by}, \by \bigr)
\in (\R_-)^{2 N_{\eta} N_{\cB}}$; note that to satisfy the necessary and sufficient condition, it is not
necessary here to resort to mixed actions of the first player.
The index $\ell_{\by}$ is any index $\ell$ such that $\norm[\by-\by^\ell]_1 \leq \eta/N_\cB$;
such an index always exists as noted at the beginning of this proof. Indeed, one then has
in particular that for each component $b \in \cB$,
\[
\bigl| y^{\ell_{\by}}_b - y_b \bigr| \leq \norm[\by^{\ell_{\by}} - \by]_1 \leq \eta/N_\cB\,.
\]

A straightforward adaptation of the proof of Theorem~\ref{th:Bla} then yields a strategy
such that for all $\delta \in (0,1)$ and for all
strategies of the second player, with probability at least $1-\delta$,
\begin{equation}
\label{eq:apprcalibr}
\sup_{\tau \geq T} \ \inf_{c \in (\R_-)^{2 N_{\eta} N_{\cB}}}
\ \norm[c - \frac{1}{\tau} \sum_{t=1}^\tau C(L_t,\by_t) ]_2 \ \leq
2M \sqrt{\frac{2}{\delta T}}\,,
\end{equation}
where $M$ is a bound in Euclidian norm over $C$, e.g., $M = 4+2\eta$.
The quantities of interest can be rewritten as
\[
\frac{1}{\tau} \sum_{t=1}^\tau C(L_t,\by_t) = \left( \frac{N_\tau(\ell)}{\tau} \bigl( \by^\ell - \ol{\by}_\tau^\ell \bigr)
- \frac{N_\tau(\ell)}{\tau} \frac{\eta}{N_{\cB}} \uo, \
\frac{N_\tau(\ell)}{\tau} \bigl( \ol{\by}_\tau^\ell - \by^\ell\bigr)
- \frac{N_\tau(\ell)}{\tau} \frac{\eta}{N_{\cB}} \uo \right)_{\! \ell \in \{ 1,\ldots,N_\eta \}},
\]
where we recall that we denoted for all $\ell$ such that $N_\tau(\ell) > 0$
the average of their corresponding mixed actions $\by_t$ by
\[
\ol{\by}_\tau^\ell = \frac{1}{N_\tau(\ell)} \sum_{t=1}^\tau
\by_t \ind_{ \{ L_t = \ell \} }\,.
\]
The projection in $\ell^2$--norm of quantity of interest onto $(\R_-)^{2 N_{\eta} N_{\cB}}$
is formed by its non-positive components, so that its square
distance to $(\R_-)^{2 N_{\eta} N_{\cB}}$ equals
\[
\inf_{c \in (\R_-)^{2 N_{\eta} N_{\cB}}}
\ \norm[c - \frac{1}{\tau} \sum_{t=1}^\tau C(L_t,\by_t) ]_2^2
\ = \sum_{\ell=1}^{N_\eta} \left( \frac{N_\tau(\ell)}{\tau} \right)^{\! 2}
\sum_{b \in \cB} \underbrace{\left( \left( y^\ell_b - \ol{y}^\ell_{\tau,b} - \frac{\eta}{N_{\cB}} \right)_{\! +}^{\! 2} +
\left( \ol{y}^\ell_{\tau,b} - y^\ell_b - \frac{\eta}{N_{\cB}} \right)_{\! +}^{\! 2} \right)}_{
= \bigl( | \ol{y}^\ell_{\tau,b} - y^\ell_b | - \eta/N_{\cB} \bigr)_{\! +}^{\! 2}}.
\]
Therefore, our target is achieved: using first that $(\,\cdot\,)_+$ is subadditive,
then applying the Cauchy-Schwarz inequality,
\begin{eqnarray*}
\sum_{\ell = 1}^{N_\eta} \frac{N_\tau(\ell)}{\tau} \Bigl( \norm[\by^\ell - \ol{\by}_\tau]_1 - \eta \Bigr)_{\!\! +}
& \leq & \sum_{\ell=1}^{N_\eta} \frac{N_\tau(\ell)}{\tau} \sum_{b \in \cB}
\left( \bigl| y^\ell_b - \ol{y}^\ell_{\tau,b} \bigr| - \frac{\eta}{N_{\cB}} \right)_{\! +} \\
& \leq & \sqrt{N_{\eta} N_{\cB}} \,\, \sqrt{
\sum_{\ell=1}^{N_\eta} \left( \frac{N_\tau(\ell)}{\tau} \right)^{\! 2}
\sum_{b \in \cB} \left( \bigl| y^\ell_b - \ol{y}^\ell_{\tau,b} \bigr| - \frac{\eta}{N_{\cB}} \right)_{\! +}^{\! 2}} \\
& \leq & 2M \sqrt{N_{\eta} N_{\cB}} \, \sqrt{\frac{2}{\delta T}}\,,
\end{eqnarray*}
where the last inequality holds, by~\eqref{eq:apprcalibr}, for all $\tau \geq T$ with probability
at least $1-\delta$. Choosing an integer $T_\delta$
sufficiently large so that
\[
2M \sqrt{N_{\eta} N_{\cB}} \, \sqrt{\frac{2}{\delta \, T_\delta}}
\leq \delta
\]
concludes the proof of the property stated in~\eqref{def:etacal}.

\section{Proof of Lemma \ref{lm:extracted}.}
\label{sec:appLemmaProof}

{$\,$}

\begin{proof}
For all $(i,j) \in \cI \times \cJ$, the quantity $H(i,j)$ is a probability distribution over the set of signals $\cH$;
we denote by $H_s(i,j)$ the probability mass that it puts on some signal $s \in \cH$.

Equation~(\ref{eq:extracted-source}) indicates that for each pair $(i,s) \in \cI \times \cH$,
\[
\sum_{t=(n-1)L+1}^{nL} \left( \frac{\ind_{ \{ S_t = s \} } \ind_{ \{ I_t = i \} }}{p_{I_t,n}} - H_s(i,J_t) \right)
\]
is a sum of $L$ elements of a martingale difference sequence,
with respect to the filtration whose $t$-th element is generated by $\bp_n$, the pairs
$(I_s, \, S_s)$ for $s \leq t$, and $J_s$ for $s \leq t+1$.
The conditional variances of the increments are bounded by
\[
\E_t \! \left[ \left( \frac{\ind_{ \{ S_t = s \} } \ind_{ \{ I_t = i \} }}{p_{I_t,n}} \right)^{\!2}
\right] \leq \frac{1}{p_{i,n}^2} \, \E_t \bigl[ \ind_{ \{ I_t = i \} } \bigr] = \frac{1}{p_{i,n}}\,;
\]
since by definition of the strategy, $\bp_n = (1-\gamma)\,\bx_n + \gamma \, \bu$, we have that
$p_{i,n} \geq \gamma/N_{\cI}$, which shows that the sum of the conditional variances is bounded by
\[
\sum_{t=(n-1)L+1}^{nL} \mathrm{Var}_t \! \left( \frac{\ind_{ \{ S_t = s \} } \ind_{ \{ I_t = i \} }}{p_{I_t,n}} \right)
\leq \frac{L N_{\cI}}{\gamma}\,.
\]
The Bernstein-Freedman inequality (see \citemor{Fre75} or \citemor{CeLuSt06}, Lemma A.1) therefore indicates that
with probability at least $1 - \delta$,
\[
\Biggl| \frac{1}{L} \sum_{t=(n-1)L+1}^{nL} \, \frac{\ind_{ \{ S_t = s \} } \ind_{ \{ I_t = i \} }}{p_{I_t,n}}
- \underbrace{\frac{1}{L} \sum_{t=(n-1)L+1}^{nL} H_s(i,J_t)}_{= \ H_s( i, \, \wh{\bq}_n )}
\Biggr| \leq \sqrt{2 \frac{N_{\cI}}{\gamma L}
\ln \frac{2}{\delta}} + \frac{1}{3} \frac{N_\cI}{\gamma L} \ln \frac{2}{\delta}\,.
\]
Therefore, by summing the above inequalities over $i \in \cI$ and $s \in \cH$,
we get (after a union bound) that with probability at least $1 - N_{\cI} N_{\cH} \delta$,
\[
\norm[ \widetilde{\sigma}_n - \tH\bigl( \wh{\bq}_n \bigr) ]_2 \leq \sqrt{N_{\cI} N_{\cH}} \left( \sqrt{\frac{2 N_{\cI}}{\gamma L}
\ln \frac{2}{\delta}} + \frac{1}{3} \frac{N_\cI}{\gamma L} \ln \frac{2}{\delta} \right).
\]
Finally, since $\wh{\sigma}_n$ is the projection in the $\ell^2$--norm of $\widetilde{\sigma}_n$
onto the convex set $\cF$, to which $\tH\bigl( \wh{\bq}_n \bigr)$ belongs, we have
that
\[
\norm[ \wh{\sigma}_n - \tH\bigl( \wh{\bq}_n \bigr) ]_2 \leq \norm[ \widetilde{\sigma}_n - \tH\bigl( \wh{\bq}_n \bigr) ]_2\,,
\]
and this concludes the proof.
\end{proof}

\bibliographystyle{plainnat}
\bibliography{Bib-Rob-Appr}

\begin{thebibliography}{29}
\providecommand{\natexlab}[1]{#1}
\providecommand{\url}[1]{\texttt{#1}}
\expandafter\ifx\csname urlstyle\endcsname\relax
  \providecommand{\doi}[1]{doi: #1}\else
  \providecommand{\doi}{doi: \begingroup \urlstyle{rm}\Url}\fi

\bibitem[Abernethy et~al.(2011)Abernethy, Bartlett, and Hazan]{ABH10}
J.~Abernethy, P.~L. Bartlett, and E.~Hazan.
\newblock Blackwell approachability and low-regret learning are equivalent.
\newblock In \emph{Proceedings of the Twenty-Fourth Annual Conference on
  Learning Theory (COLT'11)}. Omnipress, 2011.

\bibitem[Blackwell(1956{\natexlab{a}})]{Bla56}
D.~Blackwell.
\newblock An analog of the minimax theorem for vector payoffs.
\newblock \emph{Pacific Journal of Mathematics}, 6:\penalty0 1--8,
  1956{\natexlab{a}}.

\bibitem[Blackwell(1956{\natexlab{b}})]{Bla56b}
D.~Blackwell.
\newblock Controlled random walks.
\newblock In \emph{Proceedings of the International Congress of Mathematicians,
  1954, Amsterdam, vol. III}, pages 336--338, 1956{\natexlab{b}}.

\bibitem[Blum and Mansour(2007)]{BlMa07}
A.~Blum and Y.~Mansour.
\newblock From external to internal regret.
\newblock \emph{Journal of Machine Learning Research}, 8:\penalty0 1307--1324,
  2007.

\bibitem[Cesa-Bianchi and Lugosi(2006)]{CBL06}
N.~Cesa-Bianchi and G.~Lugosi.
\newblock \emph{Prediction, Learning, and Games}.
\newblock Cambridge University Press, 2006.

\bibitem[Cesa-Bianchi et~al.(2006)Cesa-Bianchi, Lugosi, and Stoltz]{CeLuSt06}
N.~Cesa-Bianchi, G.~Lugosi, and G.~Stoltz.
\newblock Regret minimization under partial monitoring.
\newblock \emph{Mathematics of Operations Research}, 31:\penalty0 562--580,
  2006.

\bibitem[Chen and White(1996)]{CW96}
X.~Chen and H.~White.
\newblock Laws of large numbers for {H}ilbert space-valued mixingales with
  applications.
\newblock \emph{Econometric Theory}, 12:\penalty0 284--304, 1996.

\bibitem[Dawid(1982)]{Dawid82}
A.P. Dawid.
\newblock The well-calibrated {B}ayesian.
\newblock \emph{Journal of the American Statistical Association}, 77:\penalty0
  605--613, 1982.

\bibitem[Foster and Vohra(1998)]{FoVo98}
D.~Foster and R.~Vohra.
\newblock Asymptotic calibration.
\newblock \emph{Biometrika}, 85:\penalty0 379--390, 1998.

\bibitem[Foster and Vohra(1999)]{FoVo99}
D.~Foster and R.~Vohra.
\newblock Regret in the on-line decision problem.
\newblock \emph{Games and Economic Behavior}, 29:\penalty0 7--36, 1999.

\bibitem[Freedman(1975)]{Fre75}
D.A. Freedman.
\newblock On tail probabilities for martingales.
\newblock \emph{Annals of Probability}, 3:\penalty0 100--118, 1975.

\bibitem[Goodman and O'Rourke(2004)]{GR04}
J.E. Goodman and J.~O'Rourke, editors.
\newblock \emph{Handbook of Discrete and Computational Geometry}.
\newblock Discrete Mathematics and its Applications. Chapman \& Hall/CRC, Boca
  Raton, FL, second edition, 2004.

\bibitem[Hart and Mas-Colell(2000)]{HM00}
S.~Hart and A.~Mas-Colell.
\newblock A simple adaptive procedure leading to correlated equilibrium.
\newblock \emph{Econometrica}, 68:\penalty0 1127--1150, 2000.

\bibitem[Hart and Mas-Colell(2001)]{HM01}
S.~Hart and A.~Mas-Colell.
\newblock A general class of adaptive strategies.
\newblock \emph{Journal of Economic Theory}, 98:\penalty0 26--54, 2001.

\bibitem[Lehrer and Solan(2007)]{LeSo07}
E.~Lehrer and E.~Solan.
\newblock Learning to play partially-specified equilibrium.
\newblock Mimeo, 2007.

\bibitem[Lugosi et~al.(2008)Lugosi, Mannor, and Stoltz]{LMS08}
G.~Lugosi, S.~Mannor, and G.~Stoltz.
\newblock Strategies for prediction under imperfect monitoring.
\newblock \emph{Mathematics of Operations Research}, 33:\penalty0 513--528,
  2008.
\newblock An extended abstract was presented at COLT'07.

\bibitem[Mannor and Shimkin(2003)]{MaSh03}
S.~Mannor and N.~Shimkin.
\newblock On-line learning with imperfect monitoring.
\newblock In \emph{Proceedings of the Sixteenth Annual Conference on Learning
  Theory (COLT'03)}, pages 552--567. Springer, 2003.

\bibitem[Mannor and Shimkin(2008)]{MShimkin08Regret}
S.~Mannor and N.~Shimkin.
\newblock Regret minimization in repeated matrix games with variable stage
  duration.
\newblock \emph{Games and Economic Behavior}, 63\penalty0 (1):\penalty0
  227--258, 2008.

\bibitem[Mannor and Stoltz(2010)]{MaSt10}
S.~Mannor and G.~Stoltz.
\newblock A geometric proof of calibration.
\newblock \emph{Mathematics of Operations Research}, 35:\penalty0 721--727,
  2010.

\bibitem[Mannor et~al.(2009)Mannor, Tsitsiklis, and Yu]{MannorTsitsiklisYu09}
S.~Mannor, J.~Tsitsiklis, and J.~Y. Yu.
\newblock Online learning with sample path constraints.
\newblock \emph{Journal of Machine Learning Research}, 10\penalty0
  (Mar):\penalty0 569--590, 2009.

\bibitem[Mertens et~al.(1994)Mertens, Sorin, and Zamir]{MeSoZa94}
J.-F. Mertens, S.~Sorin, and S.~Zamir.
\newblock Repeated games.
\newblock Technical Report no.\ 9420, 9421, 9422, Universit{\'e} de
  Louvain-la-Neuve, 1994.

\bibitem[Perchet(2009)]{Per09}
V.~Perchet.
\newblock Calibration and internal no-regret with random signals.
\newblock In \emph{Proceedings of the Twentieth International Conference on
  Algorithmic Learning Theory (ALT'09)}, pages 68--82, 2009.

\bibitem[Perchet(2011{\natexlab{a}})]{Per11}
V.~Perchet.
\newblock Approachability of convex sets in games with partial monitoring.
\newblock \emph{Journal of Optimization Theory and Applications}, 149:\penalty0
  665--677, 2011{\natexlab{a}}.

\bibitem[Perchet(2011{\natexlab{b}})]{Per11JMLR}
V.~Perchet.
\newblock Internal regret with partial monitoring calibration-based optimal
  algorithms.
\newblock \emph{Journal of Machine Learning Research}, 2011{\natexlab{b}}.
\newblock In press.

\bibitem[Perchet and Quincampoix(2011)]{PQ12}
V.~Perchet and M.~Quincampoix.
\newblock On an unified framework for approachability in games with or without
  signals.
\newblock Mimeo, 2011.

\bibitem[Piccolboni and Schindelhauer(2001)]{PiSc01}
A.~Piccolboni and C.~Schindelhauer.
\newblock Discrete prediction games with arbitrary feedback and loss.
\newblock In \emph{Proceedings of the Fourteenth Annual Conference on
  Computational Learning Theory (COLT'01)}, pages 208--223, 2001.

\bibitem[Rakhlin et~al.(2011)Rakhlin, Sridharan, and Tewari]{RST10}
A.~Rakhlin, K.~Sridharan, and A.~Tewari.
\newblock Online learning: {B}eyond regret.
\newblock In \emph{Proceedings of the Twenty-Fourth Annual Conference on
  Learning Theory (COLT'11)}. Omnipress, 2011.

\bibitem[Rambau and Ziegler(1996)]{RZ96}
J.~Rambau and G.~Ziegler.
\newblock Projections of polytopes and the generalized {B}aues conjecture.
\newblock \emph{Discrete and Computational Geometry}, 16:\penalty0 215--237,
  1996.

\bibitem[Rustichini(1999)]{Rus99}
A.~Rustichini.
\newblock Minimizing regret: {T}he general case.
\newblock \emph{Games and Economic Behavior}, 29:\penalty0 224--243, 1999.

\end{thebibliography}

\acknowledgments{
Shie Mannor was partially supported by the ISF under contract 890015 and the
Google Inter-university center for Electronic Markets and Auctions.
Vianney Perchet benefited from the support of the ANR under grant ANR-10-BLAN 0112.
Gilles Stoltz acknowledges support from
the French National Research Agency (ANR)
under grant EXPLO/RA (``Exploration--exploitation for efficient resource allocation'')
and by the PASCAL2 Network of Excellence under EC grant {no.} 506778.

An extended abstract of this paper appeared in the
\emph{Proceedings of the 24th Annual Conference on Learning Theory} (COLT'11),
JMLR Workshop and Conference Proceedings, Volume 19, pages 515--536, 2011.
}

\end{document}